\documentclass[11pt]{article}

\usepackage{authblk} 

\usepackage[dvipsnames]{xcolor}

\usepackage[margin={2.5cm,2.2cm}]{geometry} 

\usepackage{graphicx}
\usepackage{float}
\usepackage{subcaption}

\usepackage{tikz}

\usepackage{listings}

\usepackage{multirow}
\usepackage{arydshln}

\lstdefinestyle{Python}
{language=Python}

\usepackage{enumerate}
\usepackage{amsmath,amsthm,amsfonts,bm}

\usepackage{hyperref}

\theoremstyle{plain}
\newtheorem{teorema}{Theorem}[section]
\newtheorem{proposicion}[teorema]{Proposition}

\newtheorem{corolario}[teorema]{Corollary}
\newtheorem{observacion}[teorema]{Remark}
\theoremstyle{definition}

\newcommand*{\defeq}{\mathrel{\vcenter{\baselineskip0.5ex \lineskiplimit0pt
			\hbox{\scriptsize.}\hbox{\scriptsize.}}}%
	=}
\newcommand{\Rset}{\mathbb{R}}

\newcommand{\bv}{\overline{v}}
\newcommand{\bu}{\overline{u}}
\newcommand{\dt}{\partial_t}
\renewcommand{\div}{\nabla\cdot}
\newcommand{\bbeta}{\boldsymbol{\beta}}
\newcommand{\FF}{\boldsymbol{F}}
\newcommand{\Fn}{\boldsymbol{\Phi}}
\newcommand{\nn}{\boldsymbol{n}}
\newcommand{\vv}{\boldsymbol{v}}

\newcommand{\vK}{v_{K}}
\newcommand{\vKs}{v_{K^*}}
\newcommand{\vL}{v_{L}}
\newcommand{\uK}{u_{K}}
\newcommand{\uKs}{u_{K^*}}
\newcommand{\uL}{u_{L}}

\newcommand{\vmin}{v_{min}}
\newcommand{\vmax}{v_{max}}

\def\escalar#1#2{\left(#1,#2\right)}

\def\escalarL#1#2{\escalar{#1}{#2}_{L^2(\Omega)}}
\def\escalarLd#1#2{\escalar{#1}{#2}_{\left(L^2(\Omega)\right)^d}}

\def\norma#1{\| #1\|  }
\def\normaL#1{\norma{#1}_{L^2(\Omega)}}
\def\normaLd#1{\norma{#1}_{\left(L^2(\Omega)\right)^d}}
\def\normaLinf#1{\norma{#1}_{L^\infty (\Omega)}}

\def\salto#1{\left[\!\left[#1\right]\!\right]}
\def\media#1{\left\{\!\!\left\{#1\right\}\!\!\right\}}

\def\T{\mathcal{T}}
\def\E{\mathcal{E}}
\def\S{\mathcal{S}}
\def\N{\mathbb{N}}

\def\X{\mathcal{X}}

\def\Pd{\mathbb{P}^{\text{disc}}}
\def\Pc{\mathbb{P}^{\text{cont}}}

\def\aupw#1#2#3{a_h^{\text{upw}}(#1;#2,#3)}

\def\Pe{\mathsf{Pe}}

\title{\textbf{
An upwind DG scheme preserving the maximum principle for the convective Cahn-Hilliard model}}

\makeatletter
\def\@fnsymbol#1{\ensuremath{\ifcase#1\or \dagger\or \ddagger\or \mathsection\or * \or\mathparagraph\or *\or **\or \dagger\dagger \or \ddagger\ddagger \else\@ctrerr\fi}}
\makeatother

\author{Daniel Acosta-Soba\thanks{Departamento de Matemáticas, Universidad de Cádiz, Puerto Real, 11510 Cádiz, Spain -- Email: \texttt{daniel.acosta@uca.es}} \thanks{Department of Mathematics, University of Tennessee at Chattanooga, Chattanooga, TN 37403, USA}~,
~Francisco Guillén-González\thanks{Departamento de Ecuaciones Diferenciales y Análisis Numérico \& IMUS, Universidad de Sevilla, 41012 Seville, Spain -- Email: \texttt{guillen@us.es}}~,
~J. Rafael Rodríguez-Galván\thanks{Departamento de Matemáticas, Universidad de Cádiz, Puerto Real, 11510 Cádiz, Spain -- Email: \texttt{rafael.rodriguez@uca.es} -- Corresponding author}}

\begin{document}

\bibliographystyle{abbrv} 

\maketitle

\begin{abstract}

The design of numerical approximations of the Cahn-Hilliard model preserving the maximum principle is a challenging problem, even more if considering additional transport terms. In this work we present a new upwind Discontinuous Galerkin scheme for the convective Cahn-Hilliard model with degenerate mobility which preserves the maximum principle and prevents non-physical spurious oscillations. Furthermore, we show some numerical experiments in agreement with the previous theoretical results. Finally, numerical comparisons with other schemes found in the literature are also carried out.

\end{abstract}


\section{Introduction}\label{sec1}

This paper is concerned with the development of Discontinuous Galerkin (DG)
numerical schemes for the following convective Cahn-Hilliard (CCH) problem (written as a second order system): Given  an incompressible velocity field
$\vv\in C(\overline\Omega)^d$ with $\nabla\cdot\vv=0$ in $\Omega$,
such that $\vv\cdot\nn=0$ on $\partial\Omega$, find two real valued functions, the phase $u$ and the chemical potential $\mu$, defined in
$\Omega\times [0,T]$ such that:
\begin{subequations}
	\label{problema:cahn-hilliard+adveccion}
	\begin{align}
		\label{eq:cahn-hilliard+adveccion_u}
		\partial_t u&=\frac1{\Pe}\nabla\cdot\left(M(u)\nabla\mu\right)- \nabla\cdot(u\vv)\quad&\text{in }\Omega\times (0,T),\\
		\label{eq:cahn-hilliard+adveccion_mu}
		\mu&=F'(u)-\varepsilon^2\Delta u\quad&\text{in }\Omega\times (0,T),\\
		\nabla u\cdot \boldsymbol{n}&=\left(M(u)\nabla\mu-u\vv\right)\cdot \boldsymbol{n}=0 \quad &\text{on }\partial\Omega\times (0,T),\\
		\label{eq:CI.cahn-hilliard+adveccion}
		u(0)&=u_0\quad&\text{in }\Omega.
	\end{align}
\end{subequations}
The phase field variable $u$ localizes the two different phases at the values $u=0$ and $u=1$, while the interface occurs when $0<u<1$.
Here $\Omega$ is a bounded
polygonal domain in $\mathbb R^d$
($d=2$ or $3$ in practice), with boundary $\partial\Omega$ whose
outward-pointing unit normal is denoted $\nn$. Parameters are  $\Pe>0$  the
Péclet number of the flow and $\varepsilon>0$ related to the width interface. For simplicity, we take $\Pe=1$. The given initial phase is denoted by $u_0$.
The classical Cahn–Hilliard equation (CH) corresponds to the case without convection, i.e.
$\vv=0$.

We consider the double-well Ginzburg-Landau potential
$$F(u)=\frac{1}{4}u^2(1-u)^2$$ and the \textit{degenerate mobility}
\begin{equation}
	M(u)=u(1-u).\label{eq:mobility}
\end{equation}
This type of  degenerate mobility, vanishing at the pure phases $u=0$ and $u=1$,  implies the following maximum principle property~(see \cite{Elliott-Garcke_ChanHilliard_1996} for
the CH equation): if $0\le u_0\le 1$ in $\overline\Omega$ then  $0\le u(\cdot,t)\le 1$ in $\overline\Omega$ for $t\in (0,T)$. This property {does not hold} for constant mobility, and in general for fourth-order parabolic equations.
We are concerned in numerical schemes maintaining this property at the discrete level.

Cahn-Hilliard equation was originally
introduced in~\cite{cahn_hilliard_1958, Cahn_1795} as a phenomenological
model of phase separation in a binary alloy and it has been
successfully applied in different contexts as a model which
characterizes phase segregation and interface dynamics. Applications
include tumor
tissues~\cite{wu_zwieten_vanDerZee_2014,wise-lowengrub-frieboes-cristini_Tumor-2008},
image processing~\cite{bertozzi_inpainting_2007} and multi-phase fluid
systems, see e.g.\ the review~\cite{kim_review_phase-field_fluid_2012}.
The dynamics of this equation comprises a first stage in which a rapid
separation process takes place, leading to the creation of interfaces,
followed by a
{second stage where aggregation and development of bulk phases separated by thin diffuse interfaces take place}.
These two phenomena
(separation and aggregation) are characterized by different temporal
and spatial scales which, together with the non-linear and
fourth-order nature of CH, makes efficiently solving this equation an
interesting computational challenge.

The interface is represented as a layer of small thickness of order
$\varepsilon$ and the auxiliary function $u$ (the phase-field
function) is used to localize the (pure) phases $u=0$ and $u=1$.  The
chemical potential $\mu$  is the variational derivative
of the Helmholtz free energy $E(u)$ with respect to $u$,
$\mu=\partial_u E(u)$, where
\begin{align}
	E(u)=\int_\Omega\left(\frac{\varepsilon^2}{2} \vert\nabla u\vert^2+ F(u)\right) dx.
\end{align}

Then the dynamics of CH correspond to the physical energy dissipation
of $u_t = -\div J(u)$, where
$J(u)=-M(u)\nabla \mu = -M(u)\nabla\partial_u E(u)$.  For existence of
solution for CH equation with a degenerate mobility of the
type~(\ref{eq:mobility}), besides bounds for the phase over time, we
can refer to~\cite{Elliott-Garcke_ChanHilliard_1996}. A review and
many variants can be read in~\cite{Miranville_2017_CH}.

Numerical approximation of CH {is} a research topic of great
interest due to the wide spectrum of its applications, as well as a
source of interesting mathematical challenges due to the computational
issues referred above. Time approximations include both linear
schemes~\cite{guillen-gonzalez_linear_2013} and nonlinear ones, where
a Newton method is usually employed.  Although some authors (see
e.g.~\cite{Elliott-French_Cahn-Hilliard-FEM_1989}) consider the
fourth-order equation obtained by substitution {in the equation
	(\ref{eq:cahn-hilliard+adveccion_u})} of the expression of $\mu$ given
in (\ref{eq:cahn-hilliard+adveccion_mu}), the spatial discretization
is usually based on the mixed phase field/chemical potential
formulation presented in~(\ref{problema:cahn-hilliard+adveccion}).
These equations can be approximated by classical numerical methods,
including: {(a) }Finite differences, see
e.g.~\cite{furihata-daisuke_CH-2001,cheng-wise-et_al_CH-FD_2019} for
constant mobility $M(u)\equiv 1$
and~\cite{Kim_CH-FiniteDifference_2007} for degenerate mobility
similar to~(\ref{eq:mobility}), see
also~\cite{wang-wise-steven_CrystalCH_2011} for applications in
crystallography. {(b)} Finite volumes, see
e.g.~\cite{cueto-peraire_Cahn-Hilliard_FV_2008,bailo-carrillo_Cahn-Hilliard-FV_2021}
for degenerate mobility schemes. {(c)} And, above all, finite elements see
the precursory papers~\cite{Elliott-French_Cahn-Hilliard-FEM_1989}
(constant mobility) and~\cite{Barrett-Blowey-Garcke_1999_CH}
(degenerate mobility).

In recent years, an increasing number of advances has been published exploring
the use of discontinuous Galerkin (DG) methods both for
constant~\cite{kay_styles_suli_DG_2009,Aristotelous-Karakashian-Wise_2015_CH-DG}
and degenerate
mobility~\cite{Wells_Cahn-Hilliard-DG_2006,xia_xu_etal_2007_LDG_CH,Riviere_DG_advect_CH_2018,Liu-Yin_2021_Cahn-Hilliard-DG}.  Although DG methods lead in general to more complex
algorithms and to bigger amounts of degrees of freedom, they exhibit
some benefits of which one can take advantage also in CH
equations, for instance doing mesh refining and adaptivity~\cite{Aristotelous-Karakashian-Wise_2015_CH-DG}.
In this paper we are interested in the following {relevant point}: the possibility of designing conservative schemes preserving the maximum-principle $0\le u(t,x)\le 1$ at the discrete level, even for the CCH variant of CH equations,
{namely} where a convective term models the convection or transport of the phase-field.

	The idea of introducing a convection term in the CH equations modeling
	a phase-field driven by a flow, arriving at
	the CCH problem~(\ref{problema:cahn-hilliard+adveccion}), arouses
	great interest. Specially, in the case where the phase field is
	coupled with fluid equations. 
	
	To design adequate numerical approximations of these CCH equations is an extremely challenging problem because it adds the hyperbolic nature of convective terms to the inherent difficulties of the CH equation. There are not many numerical methods in the literature on this topic, although some interesting contributions can be found. The first of them~\cite{badalassi_ceniceros_banerjee-multiphase_systems_2003} is {worth mentioning} for the application of high-resolution spectral Fourier schemes. We also underline the papers~\cite{boyer2017ddfv,Li-Gao-Chen_Navier-Stokes-Cahn-Hilliard-FV_2020}, based on finite volume approximations, and also~\cite{chen2016efficient,guo2017mass}, where finite difference techniques are applied for Navier-Stokes CH equations.

Some authors have recently worked in DG methods
for spatial discretization of the CCH problem. Using DG schemes
in this context is natural because they are well-suited in
convection-dominated problems, even when the Péclet number is
substantially large. For instance, the work of~\cite{kay_styles_suli_DG_2009} is focused on construction and
convergence analysis of a DG method
for the CCH {equations} with constant mobility,
applying an interior penalty technique to the second order terms and
an upwind operator for discretization of the convection term.  Authors
of~\cite{Riviere_DG_advect_CH_2018} consider CCH with degenerate mobility  applying also an interior penalty to second order terms  in the mixed
form~(\ref{problema:cahn-hilliard+adveccion}) and a more elaborated
upwinding technique, based on a sigmoid function, to the convective
term.

These previous works show that DG methods are well suited for the approximation
of CCH equations, obtaining optimal convergence
order and maintaining most of the properties of the continuous
model.
But they have room for improvement in one specific question:  getting a maximum
principle for the phase in the discrete case. Although there are some works in which the maximum principle for the CH model is preserved at the discrete level using finite volumes \cite{bailo-carrillo_Cahn-Hilliard-FV_2021} and flux limiting for DG \cite{frank_bound-preserving_2020}, to the best knowledge of the authors, {no scheme has been published in which an upwind DG technique is used to obtain a discrete maximum principle property}.

Our main contribution in this paper is made in this regard: the
development of numerical schemes which guarantee punctual estimates of the phase
at the discrete level, in addition to maintaining the rest of
continuous properties of~(\ref{problema:cahn-hilliard+adveccion}). The
main idea is to introduce an upwind DG discretization
associated to the transport of the phase by the convective velocity
$\vv$, but also associated to the degenerate mobility $M(u)$. The main result in
this work is Theorem~\ref{thm:principio_del_maximo_DG_Cahn-Hilliard},
where we show that, for a piecewise constant approximation of the
phase, our DG scheme preserves the maximum principle, that is the discrete phase
is also bounded in $[0,1]$.

	Our scheme is specifically designed for the CH equation with non-singular (polynomial) chemical potential and degenerate mobility. For the CH equation with the logarithmic Flory-Huggins potential, a strict maximum principle is satisfied without the need of degenerate mobility.  In this case, designing maximum principle-preserving numerical schemes is a different issue. See, for instance, the paper~\cite{chen2019positivity} where a finite difference numerical  is proposed in this regard.

The paper is organized as follows:
In Section~\ref{sec:dg-discrretization}, we fix notation and review DG
techniques for discretization of conservative laws.
We consider the linear convection given by a
velocity field and show that the usual linear upwind DG numerical scheme with $P_0$  approximation preserves positivity in general, and the maximum principle in the divergence-free case.
In Section~\ref{sec:cahn-hilliard}, we consider the CCH problem~(\ref{problema:cahn-hilliard+adveccion}). Using
a truncated potential and a convex-splitting time discretization
we obtain a scheme satisfying an energy law, which is decreasing for the CH case ($\vv=0$). Then we introduce our fully discrete
scheme~(\ref{esquema_DG_upw_Eyre_cahn-hilliard+adveccion}) providing a
maximum principle for the discrete phase variable.
Finally, in Section~\ref{sec:numer-experiments} we present several numerical {tests}, comparing  our DG scheme  with two different space discretizations found in the literature: classical continuous $P_1$ finite elements and the SIP+upwind sigmoid DG approximation proposed in~\cite{Riviere_DG_advect_CH_2018}. We show error order tests and also we present qualitative comparisons where the maximum principle of our scheme is confirmed (and it is not conserved {by} the other {ones}).

\section{DG discretization of conservative laws}
\label{sec:dg-discrretization}

Throughout this section we {are going to} study the problem of approximating linear and nonlinear conservative laws preserving the positivity of the continuous models. To this purpose, we introduce an upwinding DG scheme which follows the ideas of the paper \cite{mazen_saad_2014} based on an upwinding finite volume method.

\subsection{Notation}
Firstly, we {are going to} set the notation that will be used from now on.
Let $\Omega\subset \Rset^d$ be a bounded polygonal domain. Then, we consider a shape-regular  triangular mesh $\T_h=\{K\}_{K\in \T_h}$ of size $h$ over $\Omega$, and we {denote as} $\E_h$ the set of the edges of $\T_h$ (faces if $d=3$) distinguishing between the \textit{interior edges} $\E_h^i$ and the \textit{boundary edges} $\E_h^b$. Then $\E_h=\E_h^i\cup\E_h^b$.

According to Figure \ref{fig:orientation_n_e} we will consider for an interior edge $e\in\E_h^i$ shared by the elements $K$ and $L$, i.e. $e=\partial K\cap\partial L$, that the associated unit normal vector $\nn_e$ is exterior to $K$ pointing to $L$. Moreover, for the boundary edges $e\in\E_h^b$, the unit normal vector $\nn_e$ will be pointing outwards of the domain $\Omega$.

\begin{figure}[h]
	\centering
	\begin{tikzpicture}
		\draw (0,0) -- (3,0) -- (4,3) -- (1,3) -- cycle;
		\draw (3,0) -- (1,3);
		\node at (0.8,0.6) {$K$};
		\node at (2.3,2.6) {$L$};
		\node at (1.4,2) {$e$};
		\draw[line width=1pt,-stealth] (2,1.5)--(2.75,2) node[right]{$\nn_e$};
	\end{tikzpicture}
	\caption{Orientation of the unit normal vector $\nn_e$.}
	\label{fig:orientation_n_e}
\end{figure}
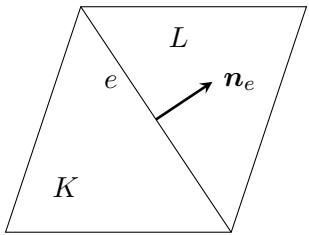

Therefore, we can define the \textit{average} $\media{\cdot}$ and the \textit{jump} $\salto{\cdot}$ of a scalar function $v$ on an edge $e\in\E_h$ as follows:

\begin{align*}
	\media{v}&\defeq
	\begin{cases}
		\dfrac{\vK+\vL}{2}&\text{if } e=\partial K\cap\partial L\in\E_h^i\\
		\vK&\text{if }e=\partial K\in\E_h^b
	\end{cases},\\
	\salto{v}&\defeq
	\begin{cases}
		\vK-\vL&\text{if } e=\partial K\cap\partial L\in\E_h^i\\
		\vK&\text{if }e=\partial K\in\E_h^b
	\end{cases}.
\end{align*}

We will denote by $\Pd_k(\T_h)$ and $\Pc_k(\T_h)$ the spaces of  discontinuous and continuous functions, respectively, whose restriction to the elements $K$ of $\T_h$ are polynomials of degree $k\ge 0$.

Moreover, we will take an equispaced partition $0=t_0<t_1<\cdots<t_N=T$ of the time domain $[0,T]$ with $\Delta t=t_{m+1}-t_m$ the time step.

Finally we set the following notation for the positive and negative parts of a scalar function $v$:
$$
v_\oplus\defeq\frac{\vert v\vert +v}{2}=\max\{v,0\},
\quad
v_{\ominus}\defeq\frac{\vert v\vert -v}{2}=-\min\{v,0\},
\quad
v=v_\oplus - v_\ominus.
$$

\subsection{Conservative laws}

We consider the following non-linear conservative problem for $v$:
\begin{equation}
	\label{eq:conserv.1}
	v_t + \div \FF(v) = 0 \quad\text{en } \Omega\times (0,T),
\end{equation}
where the flux $\FF(\cdot)$ is a vectorial continuous function.

Let $S_h=\Pd_k(\T_h)$.
Multiplying by any $\bv\in \S_h$ and using the Green Formula in each element $K\in \T_h$:
\begin{equation*}
	\int_\Omega \dt v \;\bv - \sum_{K\in \T_h} \int_K \FF(v)\cdot \nabla \bv
	+ \sum_{K\in \T_h} \int_{\partial K} \big(\FF(v)\cdot \nn_K\big) \, \bv
	= 0,
\end{equation*}
where $\nn_K$ is the normal vector to $\partial K$ pointing outwards $K$.
If $v$ is a strong solution of~(\ref{eq:conserv.1}), such that $v$ is
continuous  over $e\in\E_h$,
we get:
\begin{align}
	\label{eq:conserv.2}
	\int_\Omega \dt v \;\bv - \sum_{K\in \T_h} \int_K \FF(v) \cdot\nabla \bv
	+ \sum_{e\in \E_h} \int_{e} \big(\FF(v)\cdot \nn_e\big) \, \salto\bv
	&= 0.
\end{align}
But, if we look for functions $v\in \S_h$ which are discontinuous  over $e\in\E_h$, then the term
$\FF(v)\cdot \nn_e$ in the last integral of~(\ref{eq:conserv.2}) is not well defined and we need to approach it. Hence, the concept of numerical flux is used, taking
expressions like
$\Fn(v_K, v_L, \nn_e)$ such that:
$$
\int_{e} \big(\FF(v)\cdot \nn_e\big) \, \salto\bv
\approx
\int_{e} \Fn(\vK,\vL,\nn_e) \, \salto\bv.
$$

From now on we will consider fluxes of the form
$$\FF(v)=M(v)\bbeta,
$$
with $M=M(v)\in \mathbb{R}$ and $\bbeta=\bbeta(x,t)\in \Rset^d$. Since  in this paper we are interested in studying conservative
problems defined over isolated domains, we impose from now {on} that the
transport vector $\bbeta$ satisfies  the slip condition
\begin{equation} \label{slipcondition}
	\bbeta\cdot\nn=0 \quad\text{on}\ \partial\Omega.
\end{equation}

\begin{observacion}
	We will have to take into consideration the sign of $M'(v)$, i.e. if $M(v)$ is nonincreasing or nondecreasing, in order to work out a well-suited upwinding scheme. This is due to
	\begin{align}
		\nabla\cdot\FF(v)=
		\nabla\cdot(M(v)\bbeta)=M'(v)\bbeta\cdot\nabla v + M(v)\nabla\cdot\bbeta.
	\end{align}
	where
	\begin{itemize}
		\item The first term $M'(v)\bbeta\cdot\nabla v$ is a transport of $v$ in the direction of $M'(v)\bbeta$.
		\item The second term $M(v)\nabla\cdot\bbeta$ can be seen as a reaction term with coefficient $\nabla\cdot \bbeta$.
	\end{itemize}
	It may be specially interesting the case $\bbeta=-\nabla v$, where the whole term $\nabla\cdot(M(v)\bbeta)=-\nabla\cdot(M(v)\nabla v)$ is a nonlinear diffusion term.
	
\end{observacion}

\subsection{Linear convection and positivity}
\label{seccion:conveccion_lineal}
Taking $M(v)=v$, we arrive at the linear conservative problem:
\begin{subequations}
	\label{problema:conveccion}
	\begin{align}
		\label{eq:conveccion}
		v_t+\nabla\cdot(v\bbeta)&=0&\text{in }\Omega\times(0,T),\\
		\label{ic:conveccion}
		v(0)&=v_0&\text{in }\Omega,
	\end{align}
\end{subequations}
where $\bbeta\colon\overline\Omega\to\Rset^d$ is continuous.
\begin{observacion} \label{rmrk:positividad_modelo_conveccion_lineal}
	The problem \eqref{problema:conveccion} is well-posed without imposing
		boundary conditions due to the slip
		condition~(\ref{slipcondition}) (it can be derived writing an integral expression of solution in terms of the characteristic curves, see for example \cite{pironneau1989finite}). In particular, this integral expression implies the positivity of the solution, namely,
		if $v_0\ge 0$ in $\overline\Omega$,
		then
		the solution $v$ of \eqref{problema:conveccion} satisfies $v\ge 0$ in $\overline\Omega\times(0,T)$.
\end{observacion}

For the linear problem \eqref{eq:conveccion}, we introduce
the upwind numerical flux
\begin{equation}\label{Up-flux}
	\Fn(\vK,\vL,\nn_e)\defeq (\bbeta\cdot\nn_e)_{\oplus} \vK-(\bbeta\cdot\nn_e)_{\ominus}\vL,
\end{equation}
which leads to the following discrete scheme for
\eqref{eq:conveccion}: Given $v^{m}\in \S_h$, find
$v^{m+1}\in \S_h$ solving
\begin{align}
	\label{esquema:ecuacion_conveccion_parte_pos}
	\int_\Omega\delta_t v^{m+1}\overline{v} +\aupw{\bbeta}{v^{m+1}_{\oplus}}{\overline{v}}&=0,\quad\forall\overline{v}\in \S_h,
\end{align}
where 
\begin{align}\label{Up-Bilin}
	\aupw{\bbeta}{v}{\overline{v}}\defeq &-\sum_{K\in\T_h}\int_K v (\bbeta\cdot\nabla\overline{v})
	\nonumber \\&+\sum_{e\in\E_h^i, e=K\cap L}\int_e\left( (\bbeta\cdot\nn_e)_{\oplus}\vK
	-(\bbeta\cdot\nn_e)_{\ominus} \vL\right) \salto{\bv}
\end{align}
and
$\delta_t v^{m+1}=(v^{m+1}-v^{m})/\Delta t$ denotes a discrete time derivative.

Notice that we have truncated $v$ by its positive part $v_{\oplus}$, taking into account that the solution of the continuous model \eqref{eq:conveccion} is positive.

\begin{observacion}
	The numerical flux $\Fn(\vK,\vL,\nn_e)$ giving in \eqref{Up-flux} can be rewritten as follows: $$\Fn(\vK,\vL,\nn_e)=(\bbeta\cdot\nn_e)\media{v}+\frac{1}{2}\vert\bbeta\cdot\nn_e\vert\salto{v},$$ where $(\bbeta\cdot\nn_e)\media{v}$ is a centered-flux term and $\frac{1}{2}\vert\bbeta\cdot\nn_e\vert\salto{v}$ is the upwind term.
\end{observacion}

Now we will prove that if we set $\S_h=\Pd_0(\T_h)$, the scheme \eqref{esquema:ecuacion_conveccion_parte_pos} preserves the positivity of the continuous model. In this case, since $\sum_{K\in\T_h}\int_K (\bbeta\cdot\nabla\overline{v})v=0$, the upwind term \eqref{Up-Bilin} reduces to
\begin{equation}\label{Up-Bilin-b}
	\aupw{\bbeta}{v}{\overline{v}}
	= \sum_{e\in\E_h^i, e=K\cap L} \int_e\left( (\bbeta\cdot\nn_e)_{\oplus}\vK-(\bbeta\cdot\nn_e)_{\ominus}\vL\right)\salto{\bv}.
\end{equation}

\begin{teorema}[DG scheme \eqref{esquema:ecuacion_conveccion_parte_pos} preserves positivity]
	\label{prop:positividad_ecuacion_conveccion_parte_pos}
	If $\S_h=\Pd_0(\T_h)$, then the scheme \eqref{esquema:ecuacion_conveccion_parte_pos}  preserves positivity, that is, for any $v^m\in\S_h$ with $v^m\ge 0$ in $\overline\Omega$, then any solution $v^{m+1}$ of \eqref{esquema:ecuacion_conveccion_parte_pos} satisfies $v^{m+1}\ge 0$ in $\overline\Omega$.
\end{teorema}
\begin{proof}
	Taking the following test function
	\begin{align*}
		\overline{v}=
		\begin{cases}
			(\vKs^{m+1})_{\ominus}&\text{in }K^*\\
			0&\text{out of }K^*
		\end{cases},
		\quad
		\bv\in\S_h,
	\end{align*}
	where $K^*$ is an element of $\T_h$ such that $\vKs^{m+1}=\min_{K\in\T_h}v_{K}^{m+1}$, the scheme \eqref{esquema:ecuacion_conveccion_parte_pos} becomes
	\begin{equation} \label{scheme-v+}
		\vert K^*\vert \left(\delta_t \vKs^{m+1}\right)(\vKs^{m+1})_{\ominus}=-\aupw{\bbeta}{v^{m+1}_{\oplus}}{(\vKs^{m+1})_{\ominus}}.
	\end{equation}
	
	By applying for all $L\in\T_h$, $\vL^{m+1}\ge \vKs^{m+1}$ hence in particular $(\vL^{m+1})_{\oplus}\ge (\vKs^{m+1})_{\oplus}$   (the positive part is a non-decreasing function), we deduce
	\begin{align*}
		\aupw{\bbeta&}{v^{m+1}_{\oplus}}{(\vKs^{m+1})_{\ominus}}=\\&=\sum_{e\in\E_h^i, e=K^*\cap L}\int_e\left((\bbeta\cdot\nn_{e})_{\oplus}(\vKs^{m+1})_{\oplus}-(\bbeta\cdot\nn_{e})_{\ominus}(\vL^{m+1})_{\oplus}\right)(\vKs^{m+1})_{\ominus}\\
		&\le\sum_{e\in\E_h^i, e= K^*\cap L}\int_e\left((\bbeta\cdot\nn_{e})_{\oplus}(\vKs^{m+1})_{\oplus}-(\bbeta\cdot\nn_{e})_{\ominus}(\vKs^{m+1})_{\oplus}\right)(\vKs^{m+1})_{\ominus}.
	\end{align*}
	Since $(\vKs^{m+1})_{\oplus}(\vKs^{m+1})_{\ominus}=0$ then
	$$\aupw{\bbeta}{v^{m+1}_{\oplus}}{(\vKs^{m+1})_{\ominus}}\le 0.$$
	Therefore, from \eqref{scheme-v+},
	$$
	\vert K^*\vert \left(\delta_t \vKs^{m+1}\right)(\vKs^{m+1})_{\ominus}\ge0.
	$$
	
	On the other hand,
	$$
	0\le \vert K^*\vert \left(\delta_t \vKs^{m+1}\right)(\vKs^{m+1})_{\ominus} =
	-\frac{\vert K^*\vert }{\Delta t}\left((\vKs^{m+1})_{\ominus}^2+\vKs^m(\vKs^{m+1})_{\ominus}\right) \le 0,
	$$
	hence $(\vKs^{m+1})_{\ominus}=0$. Thanks to the choice of $K^*$, we can assure that $v^{m+1}\ge0$.
\end{proof}

\begin{teorema}[Existence of DG scheme \eqref{esquema:ecuacion_conveccion_parte_pos}]
	\label{thm:existencia_sol_esquema_ecuacion_conveccion_parte_pos}
	Assume that $\S_h=\Pd_0(\T_h)$. For any $v^m\in\S_h$, there is at least one solution $v^{m+1}\in\S_h$ of the scheme \eqref{esquema:ecuacion_conveccion_parte_pos}.
\end{teorema}

\begin{proof}
	Consider the following well-known theorem:
	\vspace{-2em}
	\begin{quote}
		\begin{teorema}[Leray-Schauder fixed point theorem]
			\label{thm:Leray-Schauder}
			Let $\X$ be a Banach space and let $T\colon\X\longrightarrow\X$ be a continuous and compact operator. If the set $$\{x\in\X\colon x=\alpha \,T(x)\quad\text{for some } 0\le\alpha\le1\}$$ is bounded (with respect to $\alpha$), then $T$ has  at least one fixed point.
		\end{teorema}
		
	\end{quote}
	Given a function $w\in\Pd_0(\T_h)$ with $w\ge0$, we are going to define the operator $T\colon \Pd_0(\T_h)\longrightarrow\Pd_0(\T_h)$ such that $T(\widehat{v})=v$ where $v$ is the unique solution of the linear scheme:
	\begin{align}
		\label{esquema:ecuacion_conveccion_Leray-Schauder}
		v\in \Pd_0(\T_h), \quad 		\int_\Omega\frac{v-w}{\Delta t}\overline{v}&=-\aupw{\bbeta}{\widehat{v}_{\oplus}}{\overline{v}},\quad\forall\overline{v}\in \Pd_0(\T_h).
	\end{align}
	The idea is to use the Leray-Schauder fixed point theorem \ref{thm:Leray-Schauder} to the operator $T$.  First of all, $T$ is well defined.
	
	Secondly, we will check that $T$ is continuous. Let $\{\widehat{v}_j\}_{j\in\N}\subset\Pd_0(\T_h)$ be a sequence such that $\lim_{j\to\infty}\widehat{v}_j=\widehat{v}$. Taking into account that all norms are equivalent in $\Pd_0(\T_h)$ since it is a finite-dimensional space, the convergence $\widehat v_j\to \widehat v$ is equivalent to the convergence elementwise $ (\widehat v_j)_K\to \widehat v_K$ for every $K\in\T_h$ (this may be seen, for instance, by using the norm $\norma{\cdot}_{L^\infty(\Omega)}$). Taking limits when $j\to \infty$ in the scheme \eqref{esquema:ecuacion_conveccion_Leray-Schauder} (with $\widehat{v}\defeq\widehat{v}_j$ and $v\defeq T(\widehat v_j)$),
	and using the notion of convergence elementwise, we get that $$\lim_{j\to \infty} T(\widehat v_j)=T(\widehat v)=T\left(\lim_{j\to \infty}\widehat v_j\right),$$ hence $T$ is continuous.
	
	In particular, as $T$ is a continuous operator defined over $\Pd_0(\T_h)$, which is finite-dimensional, $T$ is also compact.
	
	Finally, let us prove that the set $$B=\{v\in\Pd_0(\T_h)\colon v=\alpha \,T(v)\quad\text{for some } 0\le\alpha\le1\}$$ is bounded. The case $\alpha=0$ is trivial so we will assume that $\alpha\in(0,1]$. If $v\in B$, then $v$ is the solution of the scheme
	\begin{align}
		\label{esquema:ecuacion_conveccion_Leray-Schauder_alpha}
		v\in \Pd_0(\T_h), \quad 		\int_\Omega\frac{v-\alpha w}{\Delta t}\overline{v}&=-\alpha\,\aupw{\bbeta}{v_{\oplus}}{\overline{v}},\quad\forall\overline{v}\in \Pd_0(\T_h).
	\end{align}
	
	Now, testing \eqref{esquema:ecuacion_conveccion_Leray-Schauder_alpha} by $\overline v=1$, we get that $$\int_\Omega v=\alpha\int_\Omega w,$$ and, since $w\ge 0$, and since it can be proved that $v\ge 0$ using the same arguments than in Theorem  \ref{prop:positividad_ecuacion_conveccion_parte_pos}, we get that $$\norma{v}_{L^1(\Omega)}\le \norma{w}_{L^1(\Omega)}.$$ Hence, since $\Pd_0(\T_h)$ is a finite-dimensional space where the norms are equivalent, we have proved that $B$ is bounded.
	Thus, using the Leray-Schauder fixed point theorem \ref{thm:Leray-Schauder}, there is a solution of the scheme \eqref{esquema:ecuacion_conveccion_parte_pos}.
\end{proof}

Let us focus on the the following linear scheme (without
truncating $v$ by its positive part, $v_{\oplus}$): Given $v^m\in\Pd_0(\T_h)$, find $v^{m+1}\in \Pd_0(\T_h)$ such that, for every $\overline{v}\in \Pd_0(\T_h)$:
\begin{align}
	\label{esquema:ecuacion_conveccion}
	\int_\Omega\delta_t v^{m+1}\overline{v} +\aupw{\bbeta}{v^{m+1}}{\overline{v}}&=0.
\end{align}
It is well-known (see e.g.~\cite{di_pietro_ern_2012} and references
therein) that \eqref{esquema:ecuacion_conveccion} has a unique solution.
Since we have shown positivity of the nonlinear truncated scheme
\eqref{esquema:ecuacion_conveccion_parte_pos},  then any solution
of \eqref{esquema:ecuacion_conveccion_parte_pos} is the unique
solution of \eqref{esquema:ecuacion_conveccion}. This argument   implies uniqueness of  \eqref{esquema:ecuacion_conveccion_parte_pos} and
positivity  of~(\ref{esquema:ecuacion_conveccion}).

\begin{corolario}[Linear DG scheme  \eqref{esquema:ecuacion_conveccion} preserves  positivity]
	For any $v^m\in\Pd_0(\T_h)$ with $v^m\ge 0$ in $\overline\Omega$, the unique solution $v^{m+1}$ of the linear scheme \eqref{esquema:ecuacion_conveccion} is positive, i.e. $v^{m+1}\ge0$ in $\overline\Omega$.
\end{corolario}

\subsection{Linear convection  with incompressible velocity and maximum principle}
\label{sec:linear-conv-incompressiblity}
Now let us focus on the particular case of the conservation problem
\eqref{problema:conveccion}
where
$\bbeta\colon\overline\Omega\to\Rset^d$ is continuous and \textbf{incompressible},
i.e. $\nabla\cdot\bbeta=0$ in $\Omega$. In this case,  the
	solution of the problem
	\eqref{problema:conveccion},
	satisfies the
	following maximum principle (it can be proved, for instance, using the characteristics method as in Remark \ref{rmrk:positividad_modelo_conveccion_lineal}):
$$\min_{\overline\Omega}{v_0}\le v\le\max_{\overline\Omega}{v_0}\quad \hbox{in $ \overline\Omega\times (0,T) .$}$$

We are going to show that
the solution of the linear scheme \ref{esquema:ecuacion_conveccion} (without truncating $v$) preserves this  maximum principle. The proof is based on the fact that, as consequence of the divergence theorem, one has
\begin{align}
	\label{propiedad_flujo_incompresible}
	\int_{\partial K}\bbeta\cdot\nn_{K}=0,\quad\forall K\in\T_h .
\end{align}

\begin{proposicion}[Linear DG \eqref{esquema:ecuacion_conveccion} preserves the maximum principle]
	\label{thm:positividad_DG_conveccion_incompresible}
	For any $v^m\in\Pd_0(\T_h)$, the solution $v^{m+1}\in \Pd_0(\T_h)$ of the scheme \eqref{esquema:ecuacion_conveccion} satisfies the maximum principle, that is $$\min_{\overline\Omega}{v^m}\le v^{m+1}\le\max_{\overline\Omega}{v^m}\quad \text{in }{\overline\Omega}.$$
\end{proposicion}
\begin{proof}
	
	First, we will prove that  $v^{m+1}\ge \min_{\overline\Omega} v^m$. Let us denote $\vmin=\min_{\overline\Omega} v^m$.
	Taking the following test function
	\begin{align*}
		\overline{v}=
		\begin{cases}
			(\vKs^{m+1}-\vmin)_{\ominus}&\text{in }K^*\\
			0&\text{out of }K^*
		\end{cases},
	\end{align*}
	where $K^*$ is an element of $\T_h$ such that the value $\vKs^{m+1}=\min_{K\in\T_h}\vK^{m+1}$, the scheme \eqref{esquema:ecuacion_conveccion} becomes
	\begin{align*}
		\vert K^*\vert \delta_t \vKs^{m+1}(\vKs^{m+1}-\vmin)_{\ominus}=-\aupw{\bbeta}{v^{m+1}}{(\vKs^{m+1}-\vmin)_{\ominus}}.
	\end{align*}
	Now, as we have chosen $K^*$ we can assure that $\vL^{m+1}\ge \vKs^{m+1}$ for all $L\in\T_h$. Hence, using \eqref{propiedad_flujo_incompresible}, one has
	\begin{align*}
		\aupw{\bbeta&}{v^{m+1}}{(\vKs^{m+1}-\vmin)_{\ominus}}=\\&=\sum_{e\in\E_h^i, e=K^*\cap L}\int_e\left((\bbeta\cdot\nn_{e})_{\oplus}\vKs^{m+1}-(\bbeta\cdot\nn_{e})_{\ominus}\vL^{m+1}\right)(\vKs^{m+1}-\vmin)_{\ominus}\\
		&\le\sum_{e\in\E_h^i, e= K^*\cap L}\int_e\left((\bbeta\cdot\nn_{e})_{\oplus}\vKs^{m+1}-(\bbeta\cdot\nn_{e})_{\ominus}\vKs^{m+1}\right)(\vKs^{m+1}-\vmin)_{\ominus}\\&=\vKs^{m+1}(\vKs^{m+1}-\vmin)_{\ominus}\sum_{e\in\E_h^i, e=K^* \cap L}\int_e(\bbeta\cdot\nn_{e})=0.
	\end{align*}
	Therefore,
	$$
	\vert K^*\vert \delta_t \vKs^{m+1}(\vKs^{m+1}-\vmin)_{\ominus}\ge0.
	$$
	Moreover,
	\begin{align*}
		0&\le
		\vert K^*\vert (\delta_t \vKs^{m+1})(\vKs^{m+1}-\vmin)_{\ominus}
		\\&=\frac{\vert K^*\vert }{\Delta t}\left((\vKs^{m+1}-\vmin)+(\vmin-\vKs^m)\right)(\vKs^{m+1}-\vmin)_{\ominus}
		\\
		&= \frac{\vert K^*\vert }{\Delta t}\left(-(\vKs^{m+1}-\vmin)_{\ominus}^2+(\vmin-\vKs^m)(\vKs^{m+1}-\vmin)_{\ominus}\right)
		\le0,
	\end{align*}
	then	we have proved that $(\vKs^{m+1}-\vmin)_{\ominus}=0$. Hence, from the choice of $K^*$, we can assure  $v^{m+1}\ge\min_{\overline{\Omega}}v^m$.
	
	Now, we will prove that  $v^{m+1}\le \max_{\overline\Omega} v^m$.
	Let us denote $\vmax=\max_{\overline\Omega} v^m$. Taking the following test function
	\begin{align*}
		\overline{v}=
		\begin{cases}
			(\vKs^{m+1}-\vmax)_{\oplus}&\text{in }K^*\\
			0&\text{out of }K^*
		\end{cases},
	\end{align*}
	where $K^*$ is an element of $\T_h$ such that the value $\vKs^{m+1}=\max_{K\in\T_h}\vK^{m+1}$ and using similar arguments to those above, we arrive at
	$$
	\vert K^*\vert \delta_t \vKs^{m+1}(\vKs^{m+1}-\vmax)_{\oplus}\le0.
	$$
	Moreover,
	\begin{align*}
		0&\le\frac{\vert K^*\vert }{\Delta t}\left((\vKs^{m+1}-\vmax)^2_{\oplus}+(\vmax-\vKs^m)(\vKs^{m+1}-\vmax)_{\oplus}\right)\\&=\frac{\vert K^*\vert }{\Delta t}\left((\vKs^{m+1}-\vmax)+(\vmax-\vKs^m)\right)(\vKs^{m+1}-\vmax)_{\oplus}\\&=\vert K^*\vert \delta_t \vKs^{m+1}(\vKs^{m+1}-\vmax)_{\oplus}\le0,
	\end{align*}
	then we have proved that $(\vKs^{m+1}-\vmax)_{\oplus}=0$. From the choice of $K^*$, we can assure that $v^{m+1}\le\max_{\overline\Omega}v^m$.
\end{proof}

\section{Cahn-Hilliard with degenerate mobility and incompressible convection}
\label{sec:cahn-hilliard}

At this point, given $\vv\colon\Omega\times (0,T)\longrightarrow\Rset^d$
a continuous incompressible velocity field satisfying the slip condition (\ref{slipcondition}), we are in position to consider the CCH problem~(\ref{problema:cahn-hilliard+adveccion}).

\begin{observacion}
	\label{rmk:ppo_maximo_CH}
	Any smooth enough solution $(u,\mu)$ of the CCH model \eqref{problema:cahn-hilliard+adveccion}
		satisfies the  maximum principle  $0\le u\le 1$ in $\overline\Omega\times (0,T)$ whenever $0\le u_0\le 1$ in $\overline\Omega$. The proof of this statement is a straightforward consequence of  Remark \ref{rmrk:positividad_modelo_conveccion_lineal}.
	\begin{itemize}
		\item To prove that $u\ge 0$, it is enough  to notice that $u$ is the solution of
		\eqref{problema:conveccion}
		with $\bbeta\defeq-(1-u)\nabla\mu+\vv$ and to  use Remark \ref{rmrk:positividad_modelo_conveccion_lineal}.
		\item To check that $u\le 1$ we make the  change of variables $w\defeq 1-u$ and, using that $\nabla\cdot\vv=0$, notice that $w$ is the  solution of
		\eqref{problema:conveccion}
		with $\bbeta\defeq (1-w)(\nabla\mu+\vv)$. Then, we just use  Remark \ref{rmrk:positividad_modelo_conveccion_lineal}.
	\end{itemize}
\end{observacion}

Owing to this maximum principle, the following $C^2(\mathbb{R})$ truncated potential {is} considered
\begin{equation} \label{truncatedPotential}
	F(u)\defeq
	\frac{1}{4} \begin{cases}
		u^2 & u<0,\\
		u^2(1-u)^2 & u\in[0,1],\\
		(u-1)^2 & u>1.
	\end{cases}
\end{equation}
This truncated potential will allow us to define a linear time discrete convex-splitting scheme satisfying an energy law (where the energy is decreasing in the case $\vv=0$), see Subsection~\ref{sec:time-semidis}.

Assume $0\le u_0\le1$ in $\overline\Omega$. The weak formulation of  problem \eqref{problema:cahn-hilliard+adveccion} consists of finding $(u,\mu)$ such that,  $u(t)\in H^1(\Omega)$,    $\mu(t)\in  H^1(\Omega)$ with $\partial_t u(t)\in H^1(\Omega)'$, $M(u(t))\nabla\mu(t)\in L^2(\Omega)$ a.e. $t\in(0,T)$, and
satisfying the following variational problem a.e. $t\in(0,T)$ for every $\overline{\mu},\overline{u}\in H^1(\Omega)$:
\begin{subequations}
	\label{problema:cahn-hilliard+adveccion_form_var}
	\begin{align}
		\label{problema:cahn-hilliard+adveccion_form_var-A}
		\langle \partial_t u(t),\overline{\mu} \rangle &=-\escalarLd{M(u(t))\nabla\mu(t)-u(t)\vv(t)}{\nabla\overline{\mu}},\\
		\escalarL{\mu(t)}{\overline{u}}&=\varepsilon^2 \escalarLd{\nabla u(t)}{\nabla\overline{u}}+\escalarL{ F'(u(t))}{\overline{u}},
	\end{align}
\end{subequations}
with the initial condition $u(0)=u_0$ in $\Omega$.

\begin{observacion} By taking $\overline{\mu}=1$ in~\eqref{problema:cahn-hilliard+adveccion_form_var-A}, any solution $u$
	of~\eqref{problema:cahn-hilliard+adveccion_form_var}
	conserves the mass, because
	$$\frac{d}{dt}\int_\Omega u(x,t)dx=0.$$
\end{observacion}

\begin{observacion}  By taking $\overline{\mu}=\mu(t)$ and $\overline{u}=\partial_t u(t) $ in \eqref{problema:cahn-hilliard+adveccion_form_var}, and adding the resulting expressions, one has that any solution $(u,\mu)$
	of \eqref{problema:cahn-hilliard+adveccion_form_var}
	satisfies the following  energy law
	\begin{align}
		\label{ley_energia_continua_cahn-hilliard+adveccion}
		\frac{d }{dt}E(u(t))&+\int_\Omega M(u(x,t))\vert\nabla\mu(x,t)\vert^2dx
		= \int_\Omega u(x,t)\, \vv(x,t) \cdot \nabla\mu(x,t) dx ,
	\end{align}
	where $E\colon H^1(\Omega)\longrightarrow\Rset$ is the  Helmholtz free energy
	\begin{align}
		\label{energia_cahn-hilliard+adveccion}
		E(u)&\defeq\int_\Omega\left(\frac{\varepsilon^2}{2}\vert\nabla u\vert^2+ F(u)\right)dx.
	\end{align}
	Indeed, by applying the chain rule we get
	\begin{align*}
		\frac{d }{dt} E(u(t))&=\left\langle\frac{\delta E}{\delta u}(u(t)) ,\partial_t u(t)\right\rangle =\int_\Omega\mu(x,t)\partial_t u(x,t)dx
		\\&=-\int_\Omega M(u(x,t))\vert\nabla\mu(x,t)\vert^2dx + \int_\Omega u(x,t)\, \vv(x,t) \cdot \nabla\mu(x,t) dx.
	\end{align*}
	In particular, in the CH case ($\vv=0$), the energy $E(u(t))$ is dissipative. Contrarily, to the best knowledge of the authors, {for} the CCH problem with $\vv\neq 0$, there is no evidence of the existence of a dissipative energy.
\end{observacion}

\subsection{Convex splitting time discretization}
\label{sec:time-semidis}
{Now we are ready to focus on a convex splitting time  discretization (of Eyre's type~\cite{eyre_1998_unconditionally} ) of~(\ref{problema:cahn-hilliard+adveccion_form_var}). Specially, we } decompose the truncated potential \eqref{truncatedPotential} as follows:
$$
F(u)= F_{\text{i}}(u)+ F_\text{e}(u),
$$
where
$$F_\text{i}(u)\defeq
\frac{3}{8} u^2 ,
\qquad
F_\text{e}(u)\defeq \frac{1}{4} \begin{cases}
	-\frac{1}{2}u^2 & u<0,\\
	u^4-2u^3-\frac{1}{2}u^2 & u\in[0,1],\\
	1-2u-\frac{1}{2}u^2 & u>1.
\end{cases}$$
It can be easily proved that $F_\text{i}(u)$ is a convex operator, which will be
treated implicitly whereas $F_\text{e}(u)$ is a concave operator that will be
treated explicitly. Then we consider the following convex
$E_\text{i}(u(t))$ and concave $E_\text{e}(u(t))$ energy terms:
\begin{align*}
E_\text{i}(u)&\defeq \frac{\varepsilon^2}{2} \int_\Omega\vert\nabla u(x)\vert^2dx+\int_\Omega F_\text{i}(u(x))dx,\\
E_\text{e}(u)&\defeq \int_\Omega F_\text{e}(u(x))dx,
\end{align*}
such that the free energy~(\ref{energia_cahn-hilliard+adveccion}) is split as $E(u)=E_\text{i}(u)+E_\text{e}(u)$.

Finally we define the following time discretization of (\ref{problema:cahn-hilliard+adveccion_form_var}):
find $u^{m+1}\in H^1(\Omega)$ and $\mu^{m+1}\in H^1(\Omega)$ such that, for every $\overline{\mu},\overline{u}\in H^1(\Omega)$:
\begin{subequations}
\label{esquema_tiempo:cahn-hilliard+adveccion}
\begin{align}
	\label{esquema_tiempo:cahn-hilliard+adveccion_1}
	\escalarL{\delta_tu^{m+1}}{\overline{\mu}}&=-\escalarLd{M(u^{m+1})_{\oplus}\nabla\mu^{m+1}-u^{m+1}\vv(t_{m+1})}{\nabla\overline{\mu}},\\
	\label{esquema_tiempo:cahn-hilliard+adveccion_2}
	\escalarL{\mu^{m+1}}{\overline{u}}&=\varepsilon^2 \escalarLd{\nabla u^{m+1}}{\nabla\overline{u}}+\escalarL{ f(u^{m+1},u^m)}{\overline{u}},
\end{align}
\end{subequations}
where we denote the convex-implicit and concave-explicit linear approximation of the potential as follows
\begin{align}
f(u^{m+1},u^m)&\defeq F_\text{i}'(u^{m+1})+F_\text{e}'(u^m)\nonumber\\&= \frac{3}{4}u^{m+1}+
\frac{1}{4}
\begin{cases}\label{def-f(u,u0)}
	-u^m, &  u^m\in(-\infty,0),\\
	4 (u^m)^3-6(u^m)^2-u^m, & u^m\in[0,1],\\
	-\left(u^m+2\right), & u^m\in(1,+\infty).
\end{cases}
\end{align}
Notice that the positive part of the mobilitity has been taken in \eqref{esquema_tiempo:cahn-hilliard+adveccion_1}, regarding the Remark \ref{rmk:ppo_maximo_CH}, in order to prevent possible overshoots of the solution $u^{m+1}$ beyond the interval $[0,1]$.

\subsubsection{Discrete energy law}
\label{sec:ley_energia}

By adding
\eqref{esquema_tiempo:cahn-hilliard+adveccion_1} and
\eqref{esquema_tiempo:cahn-hilliard+adveccion_2} for
$\overline{\mu}=\mu^{m+1}$ and $\overline{u}=\delta_t u^{m+1}$ in \eqref{esquema_tiempo:cahn-hilliard+adveccion}, we get:
\begin{multline}
\int_\Omega M(u^{m+1})_{\oplus}\vert \nabla\mu^{m+1}\vert ^2
+\varepsilon^2\escalarLd{\nabla u^{m+1}}{\delta_t\nabla u^{m+1}} \\+\escalarL{f(u^{m+1},u^m)}{\delta_t u^{m+1}}
= \int_\Omega u^{m+1} \vv(\cdot,t_{m+1}) \cdot \nabla\mu^{m+1} .
\end{multline}
Taking into account that $$\varepsilon^2\escalarLd{\nabla u^{m+1}}{\delta_t \nabla u^{m+1}}=\frac{\varepsilon^2}{2}\delta_t\int_\Omega \vert \nabla u^{m+1}\vert ^2 +\frac{k\varepsilon^2}{2}\int_\Omega\vert \delta_t\nabla u^{m+1}\vert ^2,$$
and by adding and substracting $\delta_t\int_\Omega F(u^{m+1})$,
we get the following equality
\begin{multline*}
\delta_t E(u^{m+1})+\int_\Omega M(u^{m+1})_{\oplus}\vert \nabla\mu^{m+1}\vert ^2+\frac{k\varepsilon^2}{2}\int_\Omega\vert \delta_t\nabla u^{m+1}\vert ^2\\+\escalarL{f(u^{m+1},u^m)}{\delta_t u^{m+1}}-\delta_t\int_\Omega F(u^{m+1})
= \int_\Omega u^{m+1} \vv(\cdot,t_{m+1}) \cdot \nabla\mu^{m+1},
\end{multline*}
where $E(u)$ is defined in \eqref{energia_cahn-hilliard+adveccion}.

Then, using the Taylor theorem we get
\begin{align*}
F_\text{i}(u^m)&=F_\text{i}(u^{m+1})+F_\text{i}'(u^{m+1})(u^m-u^{m+1})+\frac{F_\text{i}''(u^{m+\xi})}{2}(u^m-u^{m+1})^2,\\
F_\text{e}(u^{m+1})&=F_\text{e}(u^{m})+F_\text{e}'(u^{m})(u^{m+1}-u^{m})+\frac{F_\text{e}''(u^{m+\eta})}{2}(u^{m+1}-u^{m})^2,
\end{align*}
for certain $\xi,\eta\in(0,1)$ with $u^{m+\xi}=\xi u^{m+1}+(1-\xi)u^m$, $u^{m+\eta}=\eta u^{m+1}+(1-\eta)u^m$. Hence, adding these expressions and taking into consideration that $F(u)=F_\text{i}(u)+F_\text{e}(u)$ for every $u\in\Rset$, we arrive at
\begin{align*}
&F(u^{m+1})-F(u^m)=\\&=f(u^{m+1},u^m)(u^{m+1}-u^m)-\frac{F''_i(u^{m+\xi})-F''_e(u^{m+\eta})}{2}(u^{m+1}-u^m)^2.
\end{align*}

Furthermore, as
\begin{align*}
&\frac{F''_i(u^{m+\xi})-F''_e(u^{m+\eta})}{2k}=\\&=
\begin{cases}
	\frac{1}{2k}, & u^{m+\eta}\in(-\infty,0)\cup(1,+\infty),\\[0.33em]
	\displaystyle\frac{3-12(u^{m+\eta})^2+12(u^{m+\eta})+1}{8k}, & u^{m+\eta}\in[0,1],
\end{cases}
\\
&=\begin{cases}
	\frac{1}{2k} \geq 0, & u^{m+\eta}\in(-\infty,0)\cup(1,+\infty),\\[0.33em]
	\frac{1}{2k}\left(1-3u^{m+\eta}(u^{m+\eta}-1)\right)\geq 0, & u^{m+\eta}\in[0,1],
\end{cases}
\end{align*}
we have
$$\frac{F''_i(u^{m+\xi})-F''_e(u^{m+\eta})}{2k}\geq0,
$$
and finally
$$
\escalarL{f(u^{m+1},u^m)}{\delta_t u^{m+1}}-\int_\Omega\delta_t F(u^{m+1})\ge0.$$
Therefore, we arrive at the following result:
\begin{teorema}
Any solution
of the scheme \eqref{esquema_tiempo:cahn-hilliard+adveccion} satisfies the following \textbf{discrete energy law} \begin{multline}
	\delta_t E(u^{m+1})+\int_\Omega M(u^{m+1})_{\oplus}\vert \nabla\mu^{m+1}\vert ^2+\frac{k\varepsilon^2}{2}\int_\Omega\vert \delta_t\nabla u^{m+1}\vert ^2\\\le
	\int_\Omega u^{m+1} \vv(\cdot,t_{m+1}) \cdot \nabla\mu^{m+1}.
\end{multline}
In particular, if $\vv=0$ the time-discrete scheme is unconditionally energy stable, because $E(u^{m+1})\le E(u^{m})$.
\end{teorema}

\subsection{Fully discrete scheme}
\label{sec:esquema_completamente_discreto}

At this point we are going to introduce  the  key idea for the spatial approximation: to treat equation \eqref{eq:cahn-hilliard+adveccion_u} as a conservative problem where there are two different fluxes (linear and nonlinear). In this sense, we propose an upwind DG scheme approximating the nonlinear flux $\FF(u)=M(u)\nabla\mu$ properly.

To this aim, we take for values $v\in\mathbb{R}$ the increasing and decreasing part of $M(v)_{\oplus}$, {denoted} respectively by $M^\uparrow(v)$ and $M^\downarrow(v)$, as follows:
\begin{subequations}
\begin{align*}
	M^\uparrow(v)&=\int_0^{v}\left(\partial_s \left(M(s)_{\oplus}\right)\right)_{\oplus}ds
	=\int_0^{\min(v,1)}M'(s)_{\oplus}ds=\int_0^{\min(v,1)}(1-2s)_{\oplus}ds,\\
	M^\downarrow(v)&=-\int_0^{v}\left(\partial_s \left(M(s)_{\oplus}\right)\right)_{\ominus}ds=-\int_0^{\min(v,1)} M'(s)_{\ominus}ds\\&=-\int_0^{\min(v,1)} (1-2s)_{\ominus}ds.
\end{align*}
\end{subequations}
Therefore,
\begin{align}
\label{movilidad_partes_creciente_decreciente}
M^\uparrow(v)=
\begin{cases}
	M(v)_{\oplus} & \text{if }v\le \frac{1}{2}\\[0.2em]
	M\left(\frac{1}{2}\right) & \text{if } v>\frac{1}{2}
\end{cases},~~
M^\downarrow(v)=
\begin{cases}
	0 & \text{if }v\le \frac{1}{2}\\
	M(v)_{\oplus}-M\left(\frac{1}{2}\right) & \text{if } v>\frac{1}{2}
\end{cases}.
\end{align}

Notice that $M^\uparrow(v)+ M^\downarrow(v) = M(v)_{\oplus} $.
We define the followig generalized upwind bilinear form to be applied for the nonlinear flux $\FF(u)=M(u)\bbeta$ where now $\bbeta$ can be discontinuous over $\E_h^i$ (in fact we will take $\bbeta=\nabla\mu$):
\begin{multline}
\label{Up-Bilin_gen}
\aupw{\bbeta}{M(v)_{\oplus}}{\overline{v}}\defeq -\int_\Omega (\bbeta\cdot\nabla\overline{v})M(v)_{\oplus}\\+\sum_{e\in\E_h^i,e=K\cap L}\int_e\left((\media{\bbeta}\cdot\nn_e)_{\oplus}(M^\uparrow(\vK)+M^\downarrow(\vL))\right.\\\left.-(\media{\bbeta}\cdot\nn_e)_{\ominus}(M^\uparrow(\vL)+M^\downarrow(\vK))\right)\salto{\overline{v}}.
\end{multline}

\begin{observacion}
We refer to \eqref{Up-Bilin_gen} as a generalized bilinear form since it generalizes the definition of \eqref{Up-Bilin} considering the case where $\bbeta$ may be discontinuous. If $\bbeta$ is continuous both definitions are equivalent.
\end{observacion}

Then, we propose the following fully discrete DG+Eyre scheme (named \textbf{DG-UPW}) for the model \eqref{problema:cahn-hilliard+adveccion}:

Find $u^{m+1}\in \Pd_0(\T_h)$, with $\mu^{m+1},w^{m+1}\in \Pc_1(\T_h)$, solving
\begin{subequations}
\label{esquema_DG_upw_Eyre_cahn-hilliard+adveccion}
\begin{align}
	\label{eq:esquema_DG_upw_Eyre_cahn-hilliard+adveccion_1}
	\escalarL{\delta_tu^{m+1}}{\overline{u}}&\nonumber\\&\hspace*{-6em}+\aupw{-\nabla\mu^{m+1}}{M(u^{m+1})_{\oplus}}{\overline{u}}+\aupw{\vv(t_{m+1})}{u^{m+1}}{\overline{u}}=0,\\
	\label{eq:esquema_DG_upw_Eyre_cahn-hilliard+adveccion_2}
	\escalarL{\mu^{m+1}}{\overline{\mu}}&=\varepsilon^2 \escalarL{\nabla w^{m+1}}{\nabla\overline{\mu}}+ \escalarL{f(u^{m+1},u^m)}{\overline{\mu}},\\
	\label{eq:esquema_DG_upw_Eyre_cahn-hilliard+adveccion_w}
	\escalarL{w^{m+1}}{\overline{w}}&=
	\escalarL{u^{m+1}}{\overline{w}},
\end{align}
\end{subequations}
for all $\overline{u}\in \Pd_0(\T_h)$ and  $\overline{\mu}, \overline{w}\in \Pc_1(\T_h)$.
Following the notation of the section \ref{seccion:conveccion_lineal},
\begin{align*}
\aupw{\vv}{u}{\overline{u}}&= \sum_{e\in\E_h^i, e=K\cap L}\int_e\left( (\vv\cdot\nn_e)_{\oplus}\uK-(\vv\cdot\nn_e)_{\ominus}\uL\right)\salto{\overline{u}}.\\
\end{align*}

In this scheme we have introduced a truncation of the function $M(u)$ taking its positive part $M(u)_{\oplus}$, which is consistent as the solution of the continuous model
\eqref{problema:cahn-hilliard+adveccion}
satisfies $0\le u \le 1$.

Notice that we have introduced a new continuous variable $w\in\Pc_1(\T_h)$ in
\eqref{eq:esquema_DG_upw_Eyre_cahn-hilliard+adveccion_w}. It can be seen as a regularization of the variable $u\in\Pc_0(\T_h)$, which is  used in the diffusion term  in \eqref{eq:esquema_DG_upw_Eyre_cahn-hilliard+adveccion_2} (which corresponds to the philic  term in the energy of the model \eqref{energia_cahn-hilliard+adveccion}). In fact, both variables $w^{m+1}$ and $u^{m+1}$ are approximations of $u(t_{m+1})$.

\begin{observacion}
We are using the same notation in the fully discrete scheme \eqref{esquema_DG_upw_Eyre_cahn-hilliard+adveccion} than the one we have used in the time-discrete scheme \eqref{esquema_tiempo:cahn-hilliard+adveccion}
given in the section \ref{sec:ley_energia},  satisfying an energy law.

Nevertheless, in this case we are changing the meaning of the equations since we are treating \eqref{eq:esquema_DG_upw_Eyre_cahn-hilliard+adveccion_1} as the $u$-equation  and \eqref{eq:esquema_DG_upw_Eyre_cahn-hilliard+adveccion_2} as the $\mu$-equation, contrary to computations done to reach the energy law.
This has been done for the purpose of preserving the maximum principle in the equation \eqref{eq:esquema_DG_upw_Eyre_cahn-hilliard+adveccion_1} and adequately approximating the laplacian term of the equation \eqref{eq:esquema_DG_upw_Eyre_cahn-hilliard+adveccion_2}.
\end{observacion}

\begin{observacion}
The boundary condition $\nabla w^{m+1}\cdot\nn=0$ on $\partial\Omega\times(0,T)$ is imposed implicitly by the term $\escalarL{\nabla w^{m+1}}{\nabla\overline{\mu}}$ in \eqref{eq:esquema_DG_upw_Eyre_cahn-hilliard+adveccion_2}.
\end{observacion}

\begin{observacion}
Since $f(\cdot,u^m)$ is linear, we have the following equality of the potential term of \eqref{eq:esquema_DG_upw_Eyre_cahn-hilliard+adveccion_2}:
$$\escalarL{f(w^{m+1},u^m)}{\overline{\mu}}=\escalarL{f(u^{m+1},u^m)}{\overline{\mu}}.$$
\end{observacion}

\begin{observacion}
The scheme \eqref{esquema_DG_upw_Eyre_cahn-hilliard+adveccion} is nonlinear, hence  we will have to use an iterative procedure, the Newton's method, to approach its solution.
\end{observacion}

\begin{proposicion}
The scheme \eqref{esquema_DG_upw_Eyre_cahn-hilliard+adveccion} conserves the mass of both $u^{m+1}$ and $w^{m+1}$ variables:
$$\int_\Omega u^{m+1}=\int_\Omega u^m\quad\text{and}\quad\int_\Omega w^{m+1}=\int_\Omega w^m.$$
\end{proposicion}
\begin{proof}
Just need to take $\overline{u}=1$ in \eqref{eq:esquema_DG_upw_Eyre_cahn-hilliard+adveccion_1} and $\overline{w}=1$ in \eqref{eq:esquema_DG_upw_Eyre_cahn-hilliard+adveccion_w}.
\end{proof}

\begin{teorema}[DG \eqref{esquema_DG_upw_Eyre_cahn-hilliard+adveccion} preserves the maximum principle]
\label{thm:principio_del_maximo_DG_Cahn-Hilliard}

For any $u^m\in\Pd_0(\T_h)$ with $0\le u^m\le1$ in $\overline\Omega$, then any solution $u^{m+1}$ of \eqref{esquema_DG_upw_Eyre_cahn-hilliard+adveccion} satisfies $0\le u^{m+1}\le 1$ in $\overline\Omega$.
\end{teorema}
\begin{proof}
Firstly, we prove
that $u^{m+1}\ge 0$.  Taking the following $\Pd_0(\T_h)$ test
function
\begin{align*}
	\overline{u}=
	\begin{cases}
		(\uKs^{m+1})_{\ominus}&\text{in }K^*\\
		0&\text{out of }K^*
	\end{cases},
\end{align*}
where $K^*$ is an element of $\T_h$ such that  $\uKs^{m+1}=\min_{K\in\T_h}u_{K}^{m+1}$, equation \eqref{eq:esquema_DG_upw_Eyre_cahn-hilliard+adveccion_1} becomes
\begin{multline}\label{vf-test}
	\vert K^*\vert \delta_t \uKs^{m+1}(\uKs^{m+1})_{\ominus}=\\= - \aupw{-\nabla\mu^{m+1}}{M(u^{m+1})_{\oplus}}{\bu}-\aupw{\vv(t_{m+1})}{u^{m+1}}{\bu}.
\end{multline}

Now,  since $\uL^{m+1}\ge \uKs^{m+1}$ for all $L\in\T_h$, we can assure that
$$M^\uparrow(\uL^{m+1}) \ge M^\uparrow(\uKs^{m+1})\quad
\hbox{and} \quad
M^\downarrow(\uL^{m+1}) \le M^\downarrow(\uKs^{m+1}).
$$
Then, we can bound as follows:
\begin{align*}
	\aupw{&-\nabla\mu^{m+1}}{M(u^{m+1})_{\oplus}}{\bu}=\\
	&=\sum_{e\in\E_h^i, e=K^*\cap L}\int_e\left((\media{-\nabla\mu^{m+1}}\cdot\nn_{e})_{\oplus}(M^\uparrow(\uKs^{m+1})+M^\downarrow(\uL^{m+1}))\right.\\
	&\qquad\qquad\left.-(\media{-\nabla\mu^{m+1}}\cdot\nn_{e})_{\ominus}(M^\uparrow(\uL^{m+1})+M^\downarrow(\uKs^{m+1}))\right)(\uKs^{m+1})_{\ominus}\\
	&\le\sum_{e\in\E_h^i, e=K^*\cap L}\int_e\left((\media{-\nabla\mu^{m+1}}\cdot\nn_{e})_{\oplus}(M^\uparrow(\uKs^{m+1})+M^\downarrow(\uKs^{m+1}))\right.\\
	&\qquad\qquad\left.-(\media{-\nabla\mu^{m+1}}\cdot\nn_{e})_{\ominus}(M^\uparrow(\uKs^{m+1})+M^\downarrow(\uKs^{m+1}))\right)(\uKs^{m+1})_{\ominus}\\&=\sum_{e\in\E_h^i, e=K^*\cap L}\int_e(\media{-\nabla\mu^{m+1}}\cdot\nn_{e})\left(M(\uKs^{m+1})\right)_{\oplus}(\uKs^{m+1})_{\ominus}=0.
\end{align*}
On the other hand, applying the incompressibility of $\vv$ and proceeding as in
Section~\ref{sec:linear-conv-incompressiblity}, one has that
$$\aupw{\vv(t_{m+1})}{u^{m+1}}{\overline{u}}\le 0.$$
Therefore, from \eqref{vf-test}
$$
\vert K^*\vert \delta_t \uKs^{m+1}(\uKs^{m+1})_{\ominus}\ge0.
$$
Consequently, it is satisfied that
$$
0
\le \vert K^*\vert (\delta_t \uKs^{m+1})(\uKs^{m+1})_{\ominus}
=
-\frac{\vert K^*\vert }{\Delta t}\left((\uKs^{m+1})_{\ominus}^2+\uKs^m(\uKs^{m+1})_{\ominus}\right)
\le 0,
$$
hence, since $\uKs^m\ge 0$, we  prove that $(\uKs^{m+1})_{\ominus}=0$. Hence $u^{m+1}\ge0$.

\

Secondly, we prove  that $u^{m+1}\le 1$. Taking the following test function
\begin{align*}
	\overline{u}=
	\begin{cases}
		(\uKs^{m+1}-1)_{\oplus}&\text{in }K^*\\
		0&\text{out of }K^*
	\end{cases},
\end{align*}
where $K^*$ is an element of $\T_h$ such that  $\uKs^{m+1}=\max_{K\in\T_h}u_{K}^{m+1}$ and using similar arguments than above, we arrive at
$$
\vert K^*\vert \delta_t \uKs^{m+1}(\uKs^{m+1}-1)_{\oplus}\le0.
$$
Besides, it is satisfied that
\begin{align*}
	0&\ge
	\vert K^*\vert \delta_t \uKs^{m+1}(\uKs^{m+1}-1)_{\oplus}
	=\frac{\vert K^*\vert }{\Delta t}\left((\uKs^{m+1}-1)+(1-\uKs^m)\right)(\uKs^{m+1}-1)_{\oplus}\\
	&=
	\frac{\vert K^*\vert }{\Delta t}\left((\uKs^{m+1}-1)^2_{\oplus}+(1-\uKs^m)(\uKs^{m+1}-1)_{\oplus}\right)
	\ge0,
\end{align*}
hence we deduce that $(\uKs^{m+1}-1)_{\oplus}=0$ and, therefore, $u^{m+1}\le1$.
\end{proof}

The following result is a direct consequence of Theorem \ref{thm:principio_del_maximo_DG_Cahn-Hilliard}.
\begin{corolario}
\label{cor:principio_del_maximo_w_DG_Cahn-Hilliard}
If we use mass-lumping to compute $w^{m+1}$ in \eqref{eq:esquema_DG_upw_Eyre_cahn-hilliard+adveccion_w}, then $0\le w^{m+1}\le 1$ in $\overline\Omega$ for $m\ge0$.
\end{corolario}

\begin{teorema}
\label{thm:existencia_solucion_CH_DG-UPW}
There is at least one solution of the scheme
\eqref{esquema_DG_upw_Eyre_cahn-hilliard+adveccion}.
\end{teorema}

\begin{proof}
Given a function $z\in\Pd_0(\T_h)$ with $0\le z \le 1$, we define the map
$$T\colon \Pd_0(\T_h)\times \Pc_1(\T_h)\times\Pc_1(\T_h)\longrightarrow\Pd_0(\T_h)\times \Pc_1(\T_h)\times\Pc_1(\T_h)$$ such that $T(\widehat{u},\widehat{\mu},\widehat{w})=(u,\mu,w)\in\Pd_0(\T_h)\times\Pc_1(\T_h)\times\Pc_1(\T_h)$ is the unique solution, for every $\overline{u}\in\Pd_0(\T_h)$,  $\overline{\mu},\overline{w}\in\Pc_1(\T_h)$, of the linear (and decoupled) scheme:
\begin{subequations}
	\begin{align}
		\label{eq:esquema_lineal_Leray-Schauder_DG_upw_Eyre_cahn-hilliard+adveccion_1}
		\frac{1}{\Delta t}\escalarL{u-z}{\overline{u}}& +\aupw{\vv}{u}{\overline{u}}
		=-\aupw{-\nabla\widehat{\mu}}{M(\widehat{u})_{\oplus}}{\overline{u}},\\
		\label{eq:esquema_lineal_Leray-Schauder_DG_upw_Eyre_cahn-hilliard+adveccion_2}
		\escalarL{\mu}{\overline{\mu}}&=\varepsilon^2 \escalarL{\nabla w}{\nabla\overline{\mu}}+ \escalarL{f(u,z)}{\overline{\mu}},\\
		\label{eq:esquema_lineal_Leray-Schauder_DG_upw_Eyre_cahn-hilliard+adveccion_w}
		\escalarL{w}{\overline{w}}&
		=\escalarL{u}{\overline{w}},
	\end{align}
\end{subequations}

To check that $T$ is well defined, one may use the following steps. First, it is easy to prove that there is a unique solution $u$ of \eqref{eq:esquema_lineal_Leray-Schauder_DG_upw_Eyre_cahn-hilliard+adveccion_1} which implies that there is a unique solution $w$ of \eqref{eq:esquema_lineal_Leray-Schauder_DG_upw_Eyre_cahn-hilliard+adveccion_w} using, for instance, the Lax-Milgram theorem. Then, it is straightforward to see that the solution $\mu$ of \eqref{eq:esquema_lineal_Leray-Schauder_DG_upw_Eyre_cahn-hilliard+adveccion_2} is unique, which implies its existence as $\Pc_1(\T_h)$ is a finite-dimensional space.

It can be proved, using the notion of convergence elementwise, as it was done in Theorem \ref{thm:existencia_sol_esquema_ecuacion_conveccion_parte_pos} and taking into consideration that $\nabla\widehat\mu\in\left(\Pd_0(\T_h)\right)^d$, that the operator $T$ is continuous, and, therefore, it is compact since $\Pd_0(\T_h)$ and $\Pc_1(\T_h)$ have finite dimension.

Finally, let us prove that the set
\begin{multline*}
	B=\{(u,\mu,w)\in\Pd_0(\T_h)\times\Pc_1(\T_h)\times\Pc_1(\T_h)\colon\\ (u,\mu,w)=\alpha T(u,\mu,w)\text{ for some } 0\le\alpha\le1\}
\end{multline*}
is bounded (independent of $\alpha$). The case $\alpha=0$ is trivial so we will assume that $\alpha\in(0,1]$.

If $(u,\mu,w)\in B$, then $u\in\Pd_0(\T_h)$ is the solution, for every $\overline{u}\in \Pd_0(\T_h)$, of
\begin{align}
	\label{eq:esquema_DG_upw_Eyre_cahn-hilliard+adveccion_Leray-Schauder_1}
	\frac{1}{\Delta t}\escalarL{u-\alpha z}{\overline{u}}&
	+ \aupw{\vv}{u}{\overline{u}}
	=-\alpha\,\aupw{-\nabla\mu}{M({u})_{\oplus}}{\overline{u}}.
\end{align}
Now, testing \eqref{eq:esquema_DG_upw_Eyre_cahn-hilliard+adveccion_Leray-Schauder_1} by $\overline u=1$, we get that $$\int_\Omega u=\alpha \int_\Omega z,$$ and, since $0\le z\le 1$, and since it can be proved that $0\le u\le 1$ using the same arguments than in Theorem \ref{thm:principio_del_maximo_DG_Cahn-Hilliard}, we get that $$\norma{u}_{L^1(\Omega)}\le \norma{z}_{L^1(\Omega)}.$$

Moreover, $w\in\Pc_1(\T_h)$ is the solution of the equation
\begin{align}
	\escalarL{w}{\overline{w}}&
	=\escalarL{u}{\overline{w}}, &\forall\overline{w}\in\Pc_1(\T_h).
\end{align}
Testing with $\overline{w}=w$ and using that $u\le1$ and that the norms are equivalent in $\Pc_1(\T_h)$, we obtain
$$\normaL{w}^2
=\escalarL{u}{w}\le\norma{w}_{L^1(\Omega)}\le \vert \Omega\vert ^{1/2} \normaL{w},$$
hence $\normaL{w}\le \vert \Omega\vert ^{1/2}$ holds.

Finally, we will check that $\mu$ is bounded.
Regarding that, $\mu\in\Pc_1(\T_h)$ is the solution of
\begin{align}
	\escalarL{\mu}{\overline{\mu}}&=\varepsilon^2 \escalarL{\nabla w}{\nabla \overline{\mu}}+ \escalarL{f(u,z)}{\overline{\mu}},&\forall\overline{\mu}\in \Pc_1(\T_h),
\end{align}
by testing by $\overline{\mu}=\mu$ we get that
\begin{align*}
	\normaL{\mu}^2&\le \varepsilon^2\normaLd{\nabla w}\normaLd{\nabla \mu}+\normaL{f(u,z)}\normaL{\mu}\\
	&\le \varepsilon^2\norma{w}_{H^1(\Omega)}\norma{\mu}_{H^1(\Omega)}+\normaL{f(u,z)}\normaL{\mu}.
\end{align*}
The norms are equivalent in the finite-dimensional space $\Pc_1(\T_h)$, therefore, there is $C_{\text{cont}}\ge 0$ such that
$$\normaL{\mu}\le \varepsilon^2C_{\text{cont}}\normaL{w}+\normaL{f(u,z)}.$$
Hence, as $0\le u,z\le 1$, we know that $\normaL{f(u,z)}$ is bounded, and therefore $\normaL{\mu}$ is bounded.

Since $\Pd_0(\T_h)$ and $\Pc_1(\T_h)$ are finite-dimensional spaces where all the norms are equivalent, we have proved that $B$ is bounded.

Thus, using the Leray-Schauder fixed point theorem \ref{thm:Leray-Schauder}, there is a solution $(u,\mu,w)$ of the scheme
\eqref{esquema_DG_upw_Eyre_cahn-hilliard+adveccion}.
\end{proof}

\begin{corolario}
There is at least one solution of the following ({non-truncated}) scheme:

Find $u^{m+1}\in \Pd_0(\T_h)$ with $0\le u^{m+1}\le1$ and $\mu^{m+1},w^{m+1}\in \Pc_1(\T_h)$ with $0\le w^{m+1}\le 1$, solving
\begin{subequations}
	\label{esquema_DG_upw_Eyre_cahn-hilliard+adveccion_no_parte_pos}
	\begin{align}
		\label{eq:esquema_DG_upw_Eyre_cahn-hilliard+adveccion_no_parte_pos_1}
		\escalarL{\delta_tu^{m+1}}{\overline{u}}&\nonumber\\&\hspace*{-5em}+\aupw{-\nabla\mu^{m+1}}{M(u^{m+1})}{\overline{u}} +\aupw{\vv(t_{m+1})}{u^{m+1}}{\overline{u}}=0,\\
		\label{eq:esquema_DG_upw_Eyre_cahn-hilliard+adveccion_no_parte_pos_2}
		\escalarL{\mu^{m+1}}{\overline{\mu}}&=\varepsilon^2 \escalarL{\nabla w^{m+1}}{\nabla\overline{\mu}}+ \escalarL{f(u^{m+1},u^m)}{\overline{\mu}},\\
		\label{eq:esquema_DG_upw_Eyre_cahn-hilliard+adveccion_no_parte_pos_w}
		\escalarL{w^{m+1}}{\overline{w}}&=
		\escalarL{u^{m+1}}{\overline{w}}.
	\end{align}
\end{subequations}
for all $\overline u \in \Pd_0(\T_h)$ and $\overline{\mu}, \overline{w}\in \Pc_1(\T_h)$. Here,  we have considered $M(u^{m+1})$ instead of $M(u^{m+1})_{\oplus}$.
\end{corolario}
\begin{proof}
By Theorems \ref{thm:principio_del_maximo_DG_Cahn-Hilliard} and \ref{thm:existencia_solucion_CH_DG-UPW} we know that there is a solution of the scheme
\eqref{esquema_DG_upw_Eyre_cahn-hilliard+adveccion}
such that
$0\le u^{m}\le 1$ in $\overline\Omega$ for every $m\ge0$. Hence, $M(u^m)=M(u^m)_{\oplus}$ for every $m\ge0$, and therefore the solution of
\eqref{esquema_DG_upw_Eyre_cahn-hilliard+adveccion}
is also a solution of
\eqref{esquema_DG_upw_Eyre_cahn-hilliard+adveccion_no_parte_pos},
which moreover satisfies the discrete maximum principle.
\end{proof}

\section{Numerical experiments}
\label{sec:numer-experiments}

We now present several numerical tests in which we explore the behaviour of the new upwind DG scheme presented in this work (\textbf{DG-UPW})~(\ref{esquema_DG_upw_Eyre_cahn-hilliard+adveccion}) and compare it with two other space semidiscretizations found in the literature: {firstly a classical FEM discretization (\textbf{FEM}) and secondly the DG scheme proposed in~\cite{Riviere_DG_advect_CH_2018}, based on an SIP + (sigmoid upwind) technique, that we call (\textbf{DG-SIP)}}. We use $P_1$ piecewise polynomials {for both schemes} unless otherwise specified.

{For} the DG-UPW scheme, we use mass lumping to compute $w^{m+1}$ in \eqref{eq:esquema_DG_upw_Eyre_cahn-hilliard+adveccion_w} so that $w^{m+1}\in[0,1]$ by the Corollary \ref{cor:principio_del_maximo_w_DG_Cahn-Hilliard}. This  $w^{m+1}$, which is
a regularization of the primal variable $u^{m+1}$, is considered as the main phase-field
variable, which is used when showing the results of the numerical experiments.
Moreover, we consider $\Pe=1$ unless another value is specified.

\begin{figure}[t]
	\begin{minipage}{0.5\linewidth}
		\centering
		\textbf{No convection}
		\includegraphics[scale=0.5]{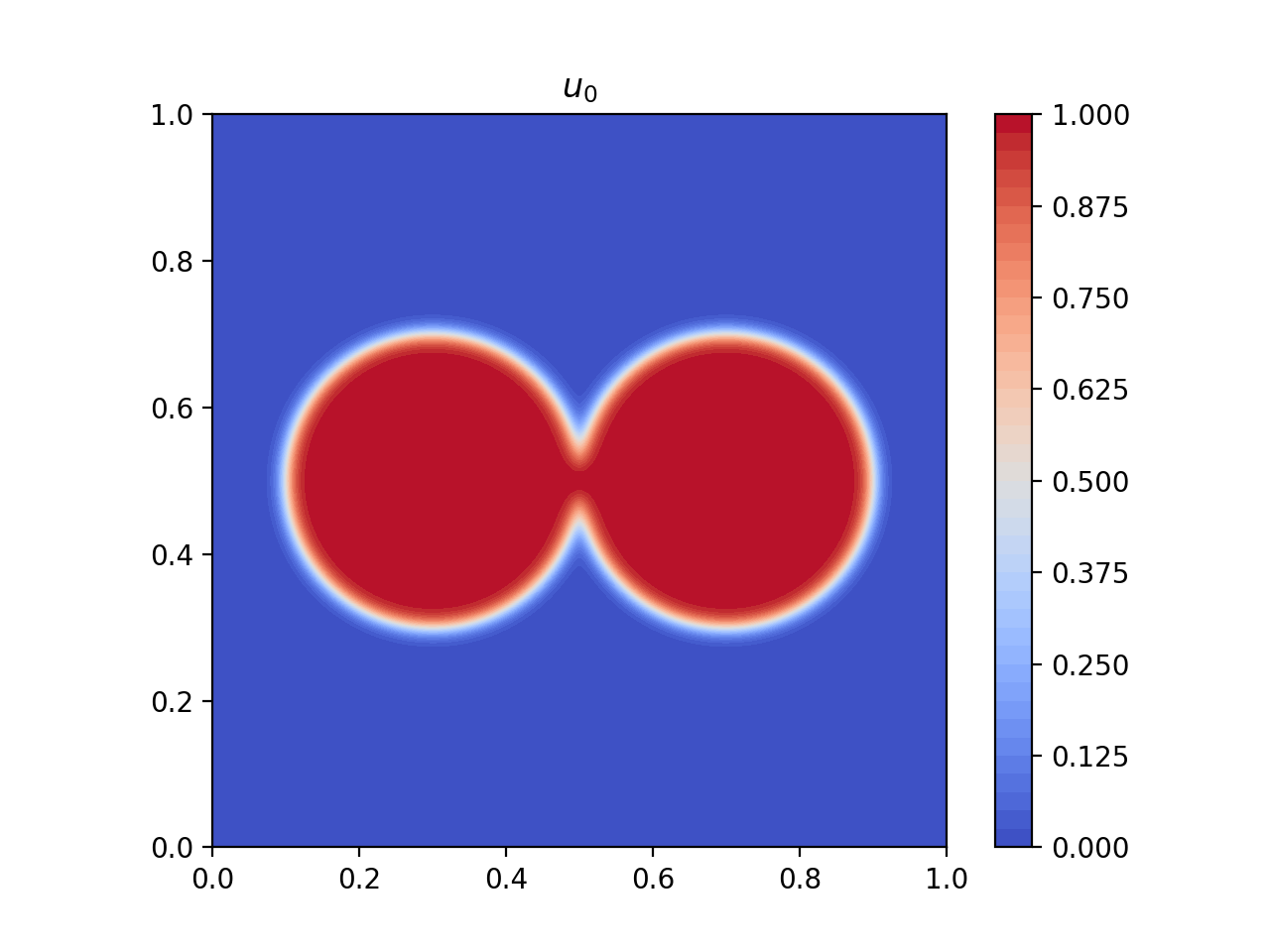}
	\end{minipage}
	\begin{minipage}{0.5\textwidth}
		\centering
		\textbf{Convection}
		\includegraphics[scale=0.5]{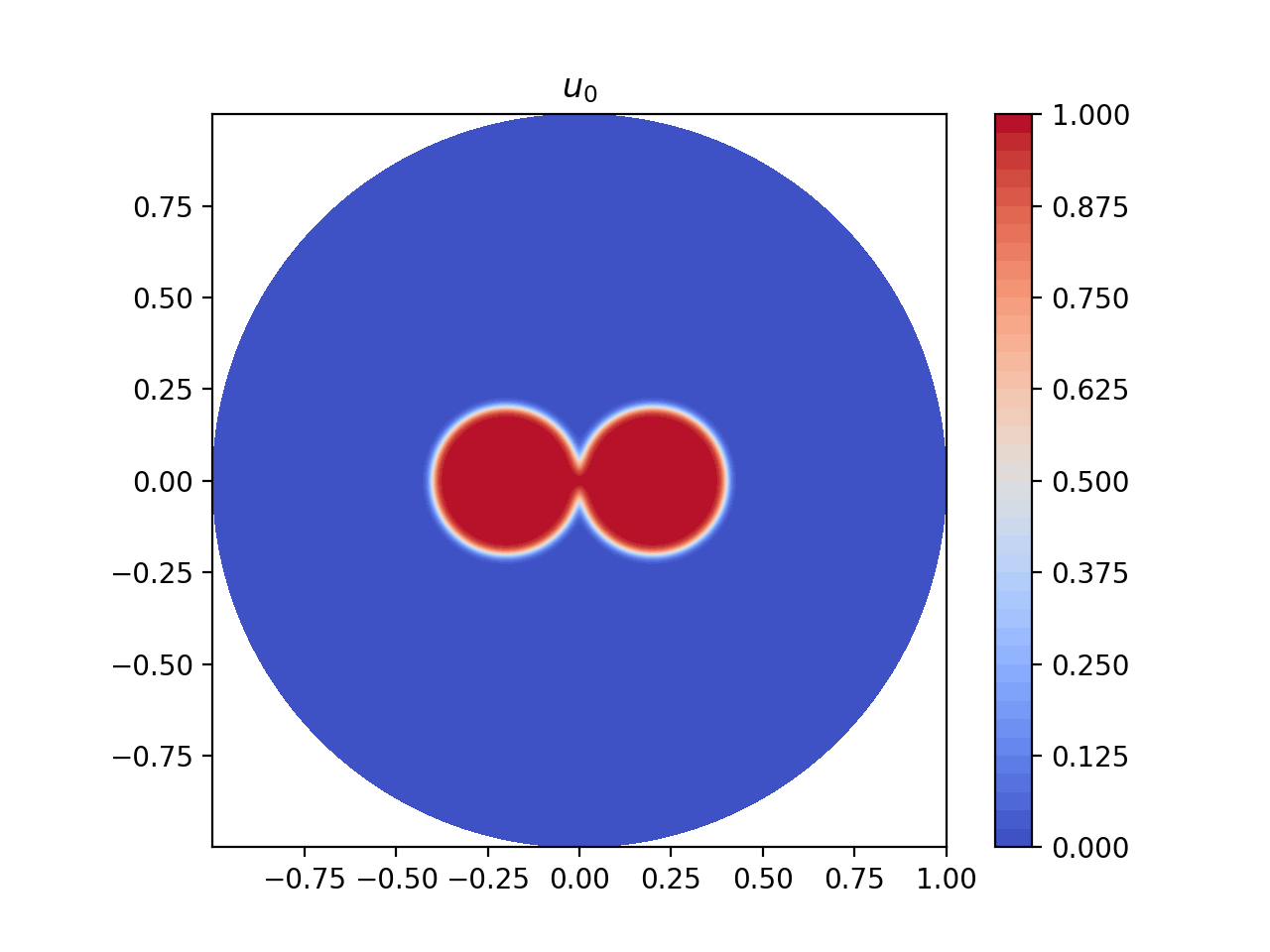}
	\end{minipage}
	\caption{Left: initial condition in the case $\vv=0$. Right: initial condition in the case $\vv=100(y,-x)$.}
	\label{fig:condicion_inicial_agregacion-analisis_errores}
\end{figure}

\begin{figure}[t]
	\centering
	\begin{minipage}{0.32\linewidth}
		\centering
		\textbf{FEM}
		\includegraphics[scale=0.095]{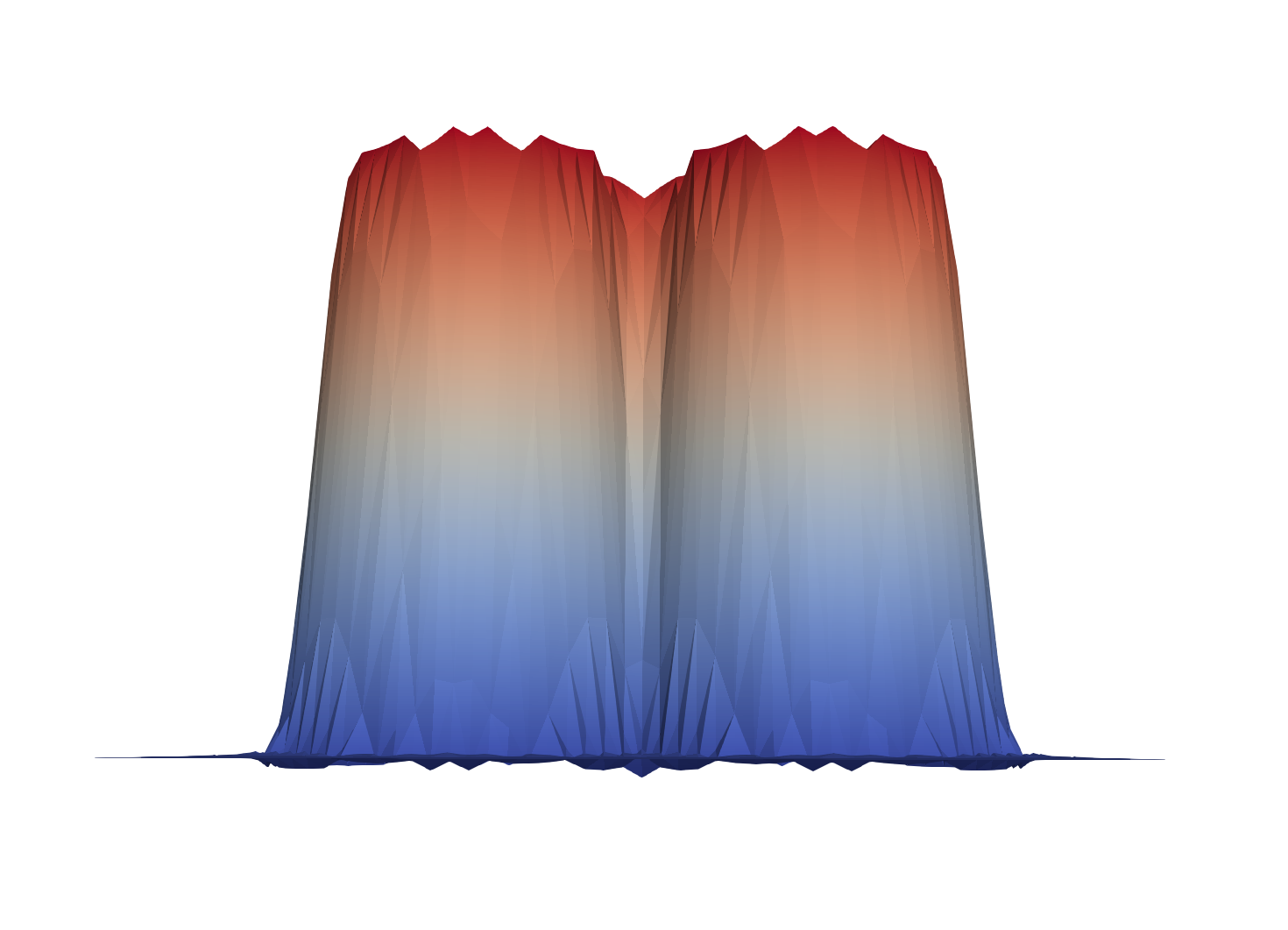}
	\end{minipage}
	\begin{minipage}{0.32\linewidth}
		\centering
		\textbf{DG-SIP}
		\includegraphics[scale=0.095]{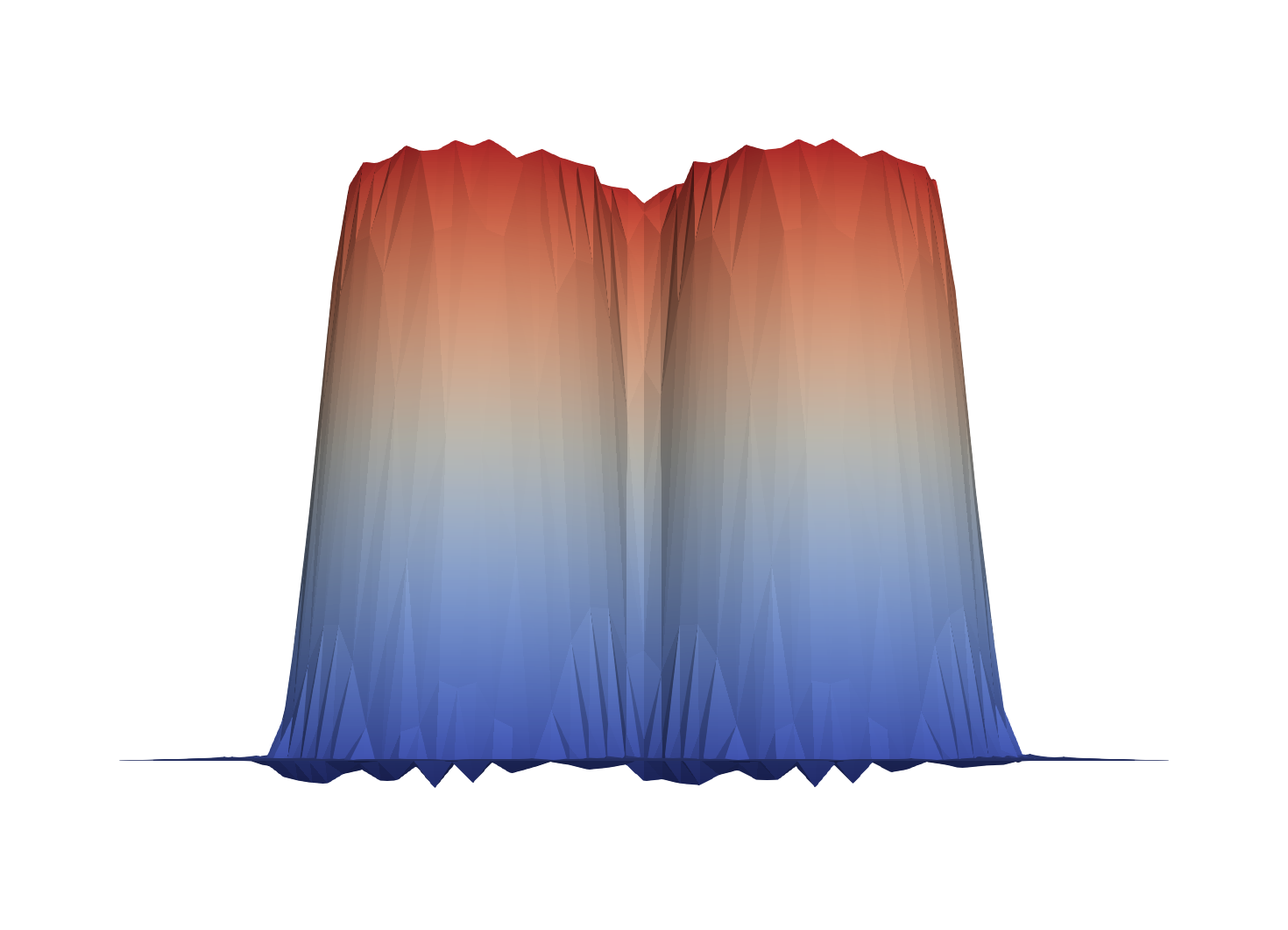}
	\end{minipage}
	\begin{minipage}{0.32\linewidth}
		\centering
		\textbf{DG-UPW}
		\includegraphics[scale=0.095]{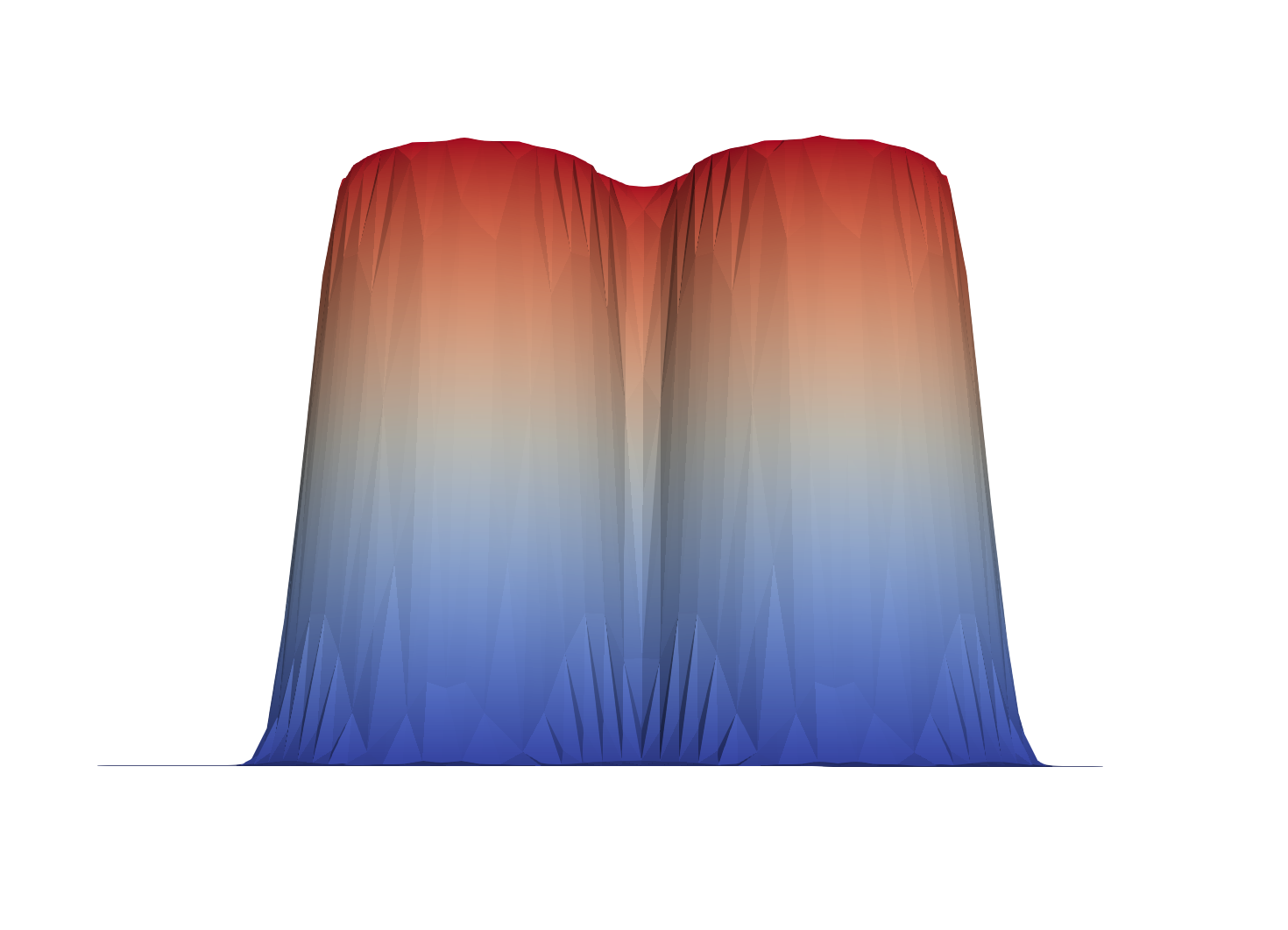}
	\end{minipage}
	\caption{Aggregation of circular regions {at} $T=0.001$ without convection, 3D view (height represents the value of the phase variable on each point of the squared domain)}
	\label{fig:comparacion_esquemas_no-conveccion}
\end{figure}

\subsection{Qualitative tests and comparisons}

Our first numerical tests are devoted to qualitative experiments about our DG-UPW scheme in rectangular and circular domains with different kinds of velocity fields. We also inspect the discrete energy and the maximum principle property, confirming that
the latter one holds for our scheme but not for the two aforementioned ones, FEM and DG-SIP.

\subsubsection{Agreggation of circular regions without convection}
\label{sec:aggregation_circular_regions}

First, we consider the Cahn–Hilliard equation without convection ($\vv=0$) in the squared domain $\Omega=(0,1)^2$ and the following initial condition (two small circles of radius $0.2$, see Figure \ref{fig:condicion_inicial_agregacion-analisis_errores}, left) :
\begin{multline}
u_0=\frac{1}{2}\left(\frac{\tanh{\left(0.2-\sqrt{(x-x_1)^2+(y-y_1)^2}\right)}}{\sqrt{2}\varepsilon}+1\right) \\+ \frac{1}{2}\left(\frac{\tanh{\left(0.2-\sqrt{(x-x_2)^2+(y-y_2)^2}\right)}}{\sqrt{2}\varepsilon}+1\right),
\label{initial_condition_1}
\end{multline}
with centers $(x_1,y_1)=(0.3,0.5)$ and $(x_2,y_2)=(0.7,0.5)$.
We take a structured mesh with $h\approx 2.8284\cdot 10^{-2}$ and run time iterations with $\Delta t=10^{-6}$.
Each iteration consists of solving a nonlinear system for computing $(u^{n+1},\mu^{n+1},w^{n+1})$, for which we {use}
Newton's method iterations, programmed on the FEniCS finite element library~\cite{FEniCS:LoggMardalEtAl2012a,Fenics:AlnaesEtAl:2015}.
For linear systems we used a MPI parallel solver (GMRES) in the computing cluster of the Universidad de C\'adiz.

\begin{figure}[t]
	\centering
	\begin{minipage}{0.49\linewidth}
		\centering
		\textbf{Maximum-Minimum}
		\includegraphics[scale=0.51]{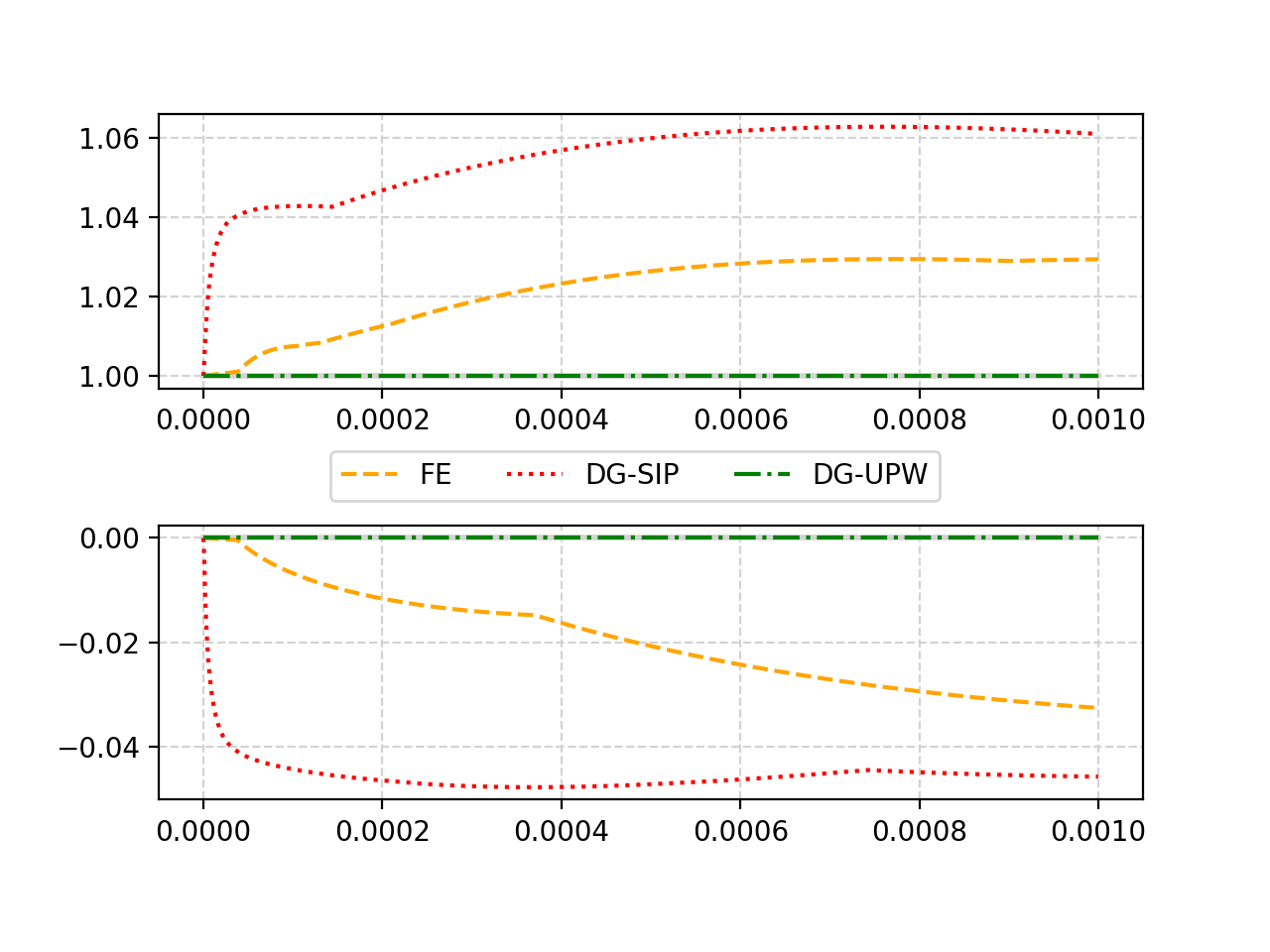}
	\end{minipage}
	\begin{minipage}{0.49\linewidth}
		\centering
		\textbf{Energy}
		\includegraphics[scale=0.5]{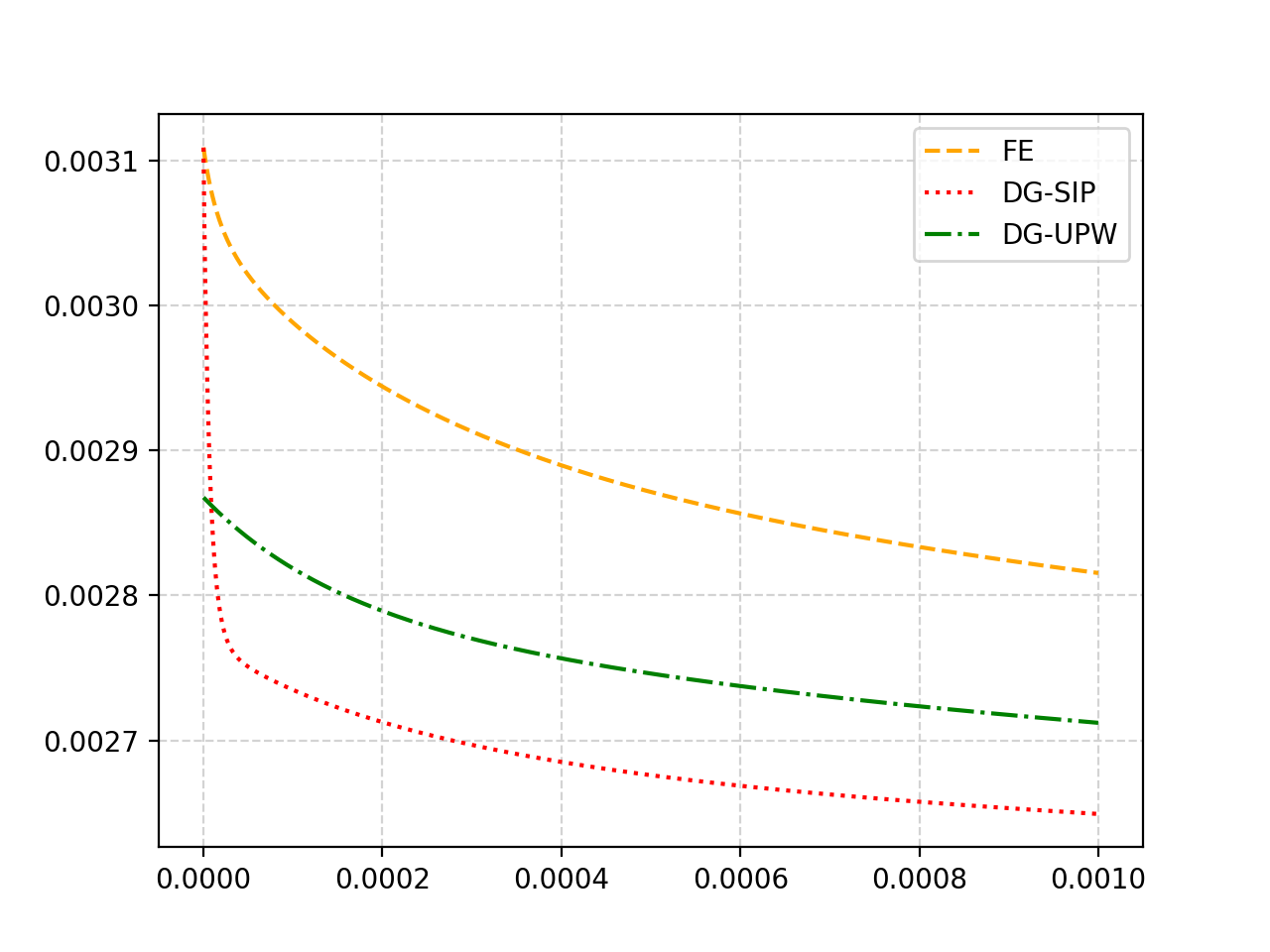}
	\end{minipage}
	\caption{Aggregation of circular phases. On the left, maximum (top) and minimum (bottom) of the phase field variable over time without convection ($\vv=0$). On the right, energy over time.}
	\label{fig:comparacion_esquemas_no-conveccion_max}
\end{figure}

In Figure~\ref{fig:comparacion_esquemas_no-conveccion} we show a 3D view of the phase field function at the time step $T=0.001$, when the aggregation process has started.
It is interesting to notice that for our upwind DG scheme \eqref{esquema_DG_upw_Eyre_cahn-hilliard+adveccion} there are no spurious oscillations  meanwhile for the FEM and DG-SIP schemes we obtain several numerical issues (vertical fluctuations in the 3D graphics).

Moreover, in the Figure \ref{fig:comparacion_esquemas_no-conveccion_max} we can clearly observe how the maximum principle is preserved by DG-UPW scheme (and not by the two other ones). Regarding the energy, we obtain a non-increasing behaviour as expected from the continuous model. An analytical proof of this property in the discrete case is left as future work.

\subsubsection{Agreggation of circular regions with convection}

Second, we define $\Omega$ as the unit ball in $\mathbb R^2$ and, again, the initial condition~(\ref{initial_condition_1}) (two small circles of radius 0.2), with centers $(x_1,y_1)=(-0.2,0)$ and $(x_2,y_2)=(0.2,0)$,
see Figure \ref{fig:condicion_inicial_agregacion-analisis_errores} (right). Moreover, {for testing the effect of convection in our scheme,} we take $\varepsilon=0.001$ and $\vv=100(y,-x)$, so that $\vv\cdot\nn=0$ on $\partial\Omega$. We take an unstructured mesh with $h\approx 4\cdot 10^{-2}$ and run time iterations with $\Delta t=10^{-3}$. Figure \ref{fig:comparacion_esquemas_conveccion} shows the values of the phase field function at different time steps. We can observe that, despite of our election of a highly significant convection term,
the results of scheme DG-UPW are qualitatively correct. Nevertheless, we can observe how the spurious oscillations become more important and the solution begins to have an unexpected behaviour when using both the FEM and DG-SIP schemes.

Concerning the maximum principle, in Figure \ref{fig:comparacion_esquemas_conveccion_max} we can see how this property is preserved  by the DG-UPW scheme~\eqref{esquema_DG_upw_Eyre_cahn-hilliard+adveccion}, while the phase field variable reaches nonphysical values, very far from $[0,1]$, when using the other aforementioned schemes. Moreover, it is interesting to observe the approximation of the steady state of the schemes, represented by the quantity $\frac{\normaLinf{u^{m+1}-u^m}}{\normaLinf{u^m}}$, which tends to $0$ in the case DG-UPW. This fact indicates that the solution converges to a stationary state, while for the other schemes it remains in an oscillatory state.

The computational time spent to obtain the results (computed sequentially) with each of these schemes is 2:32min using DG-UPW, 1:10min using FE and 3:24min using DG-SIP.

For the fairness of comparisons, we redo the tests using both FE and DG-SIP reducing the step size of the mesh, on the one hand, and using higher order polynomials, on the other hand. First, if we reduce the mesh size to $h/2\approx 2\cdot 10^{-2}$ and we use $P_1$ polynomials, the FEM scheme does not converge (the linear solver, GMRES, fails to converge) while the DG-SIP scheme gives us the results shown in Figure \ref{fig:comparacion_esquemas_conveccion_refinado} (left), requiring 35:20min to complete the computations sequentially. Second, if we keep the mesh size $h\approx 4\cdot 10^{-2}$ and we use $P_2$ polynomials, the FEM scheme does not converge either (Newton's method does not converge) while the DG-SIP scheme gives us the results shown in Figure \ref{fig:comparacion_esquemas_conveccion_refinado} (right), requiring 11min to complete the computations sequentially. Therefore, the FEM scheme does not even converge if we try to improve the results above and, while the DG-SIP scheme does converge, the results still show spurious oscillations and require a much longer computational time to be completed than those shown in Figure \ref{fig:comparacion_esquemas_conveccion} (right) using our DG-UPW scheme.

\begin{figure}[H]
	\begin{tabular}{l r}
		\rotatebox[origin=c]{90}{t=0.025} &
		\hspace*{-0.5cm}
		\begin{minipage}{0.32\textwidth}
			\centering
			\textbf{FEM}
			\includegraphics[scale=0.35]{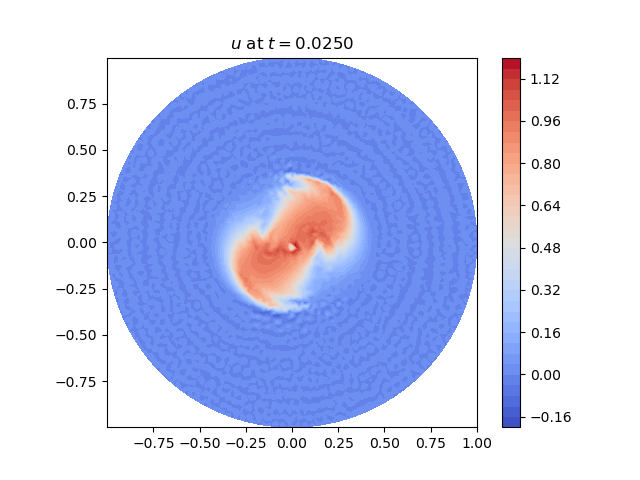}
		\end{minipage}
		\begin{minipage}{0.32\textwidth}
			\centering
			\textbf{DG-SIP}
			\includegraphics[scale=0.35]{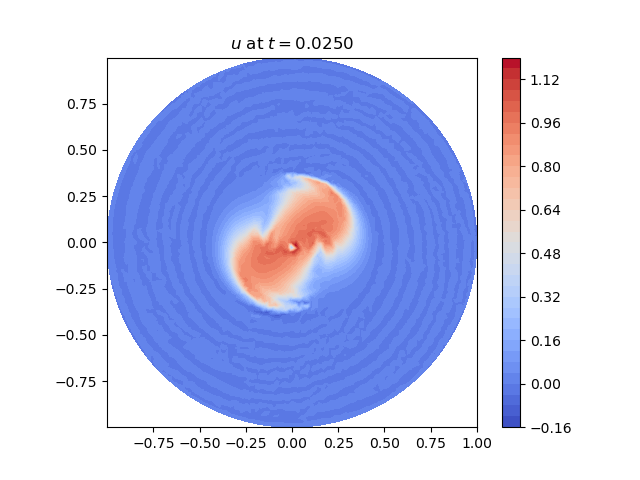}
		\end{minipage}
		\begin{minipage}{0.32\textwidth}
			\centering
			\textbf{DG-UPW}
			\includegraphics[scale=0.35]{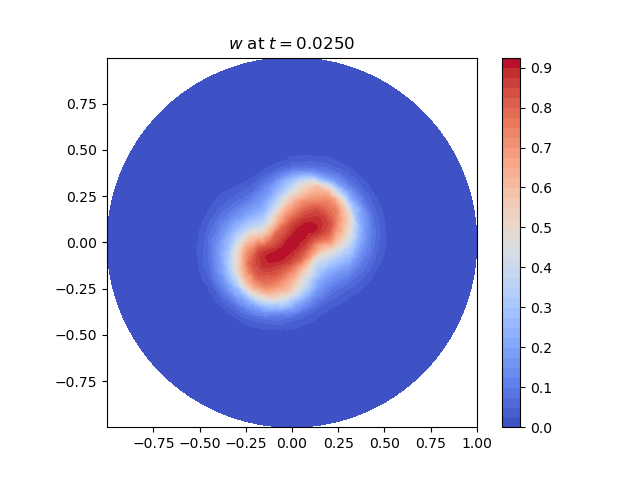}
		\end{minipage}
	\end{tabular}
	\begin{tabular}{l r}
		\rotatebox[origin=c]{90}{t=0.050} &
		\hspace*{-0.5cm}
		\begin{minipage}{0.32\textwidth}
			\centering
			\includegraphics[scale=0.35]{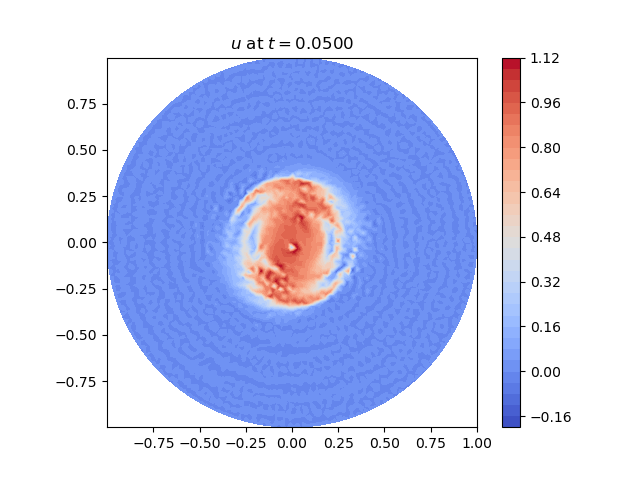}
		\end{minipage}
		\begin{minipage}{0.32\textwidth}
			\centering
			\includegraphics[scale=0.35]{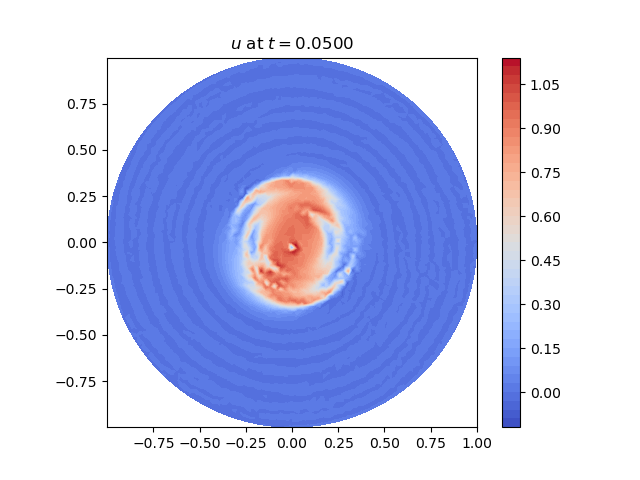}
		\end{minipage}
		\begin{minipage}{0.32\textwidth}
			\centering
			\includegraphics[scale=0.35]{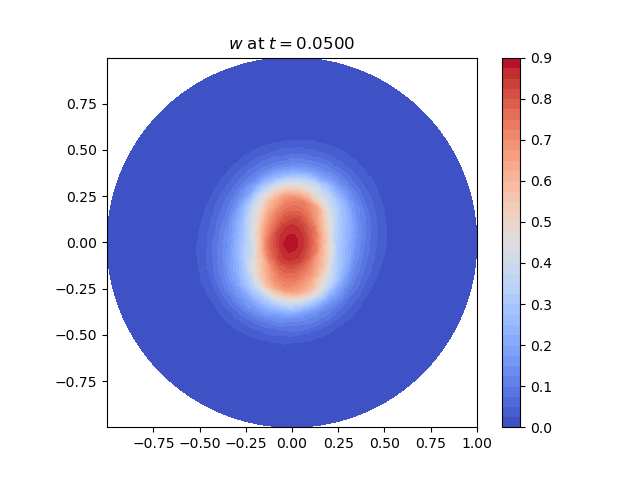}
		\end{minipage}
	\end{tabular}
	\begin{tabular}{l r}
		\rotatebox[origin=c]{90}{t=0.075} &
		\hspace*{-0.5cm}
		\begin{minipage}{0.32\textwidth}
			\centering
			\includegraphics[scale=0.35]{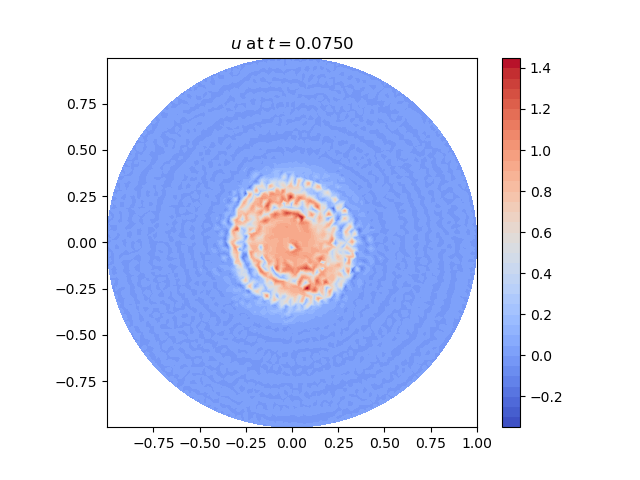}
		\end{minipage}
		\begin{minipage}{0.32\textwidth}
			\centering
			\includegraphics[scale=0.35]{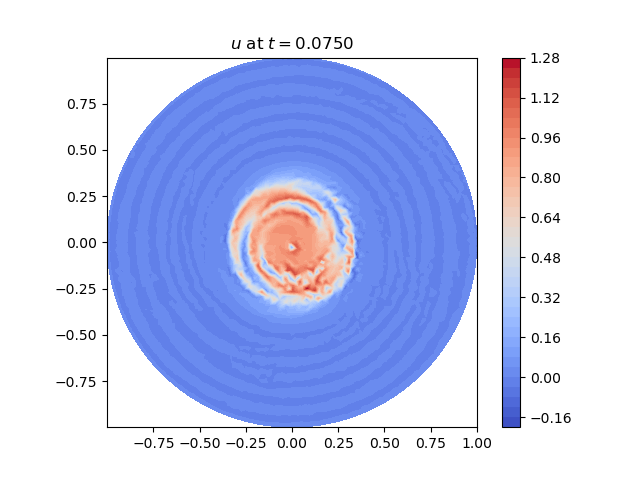}
		\end{minipage}
		\begin{minipage}{0.32\textwidth}
			\centering
			\includegraphics[scale=0.35]{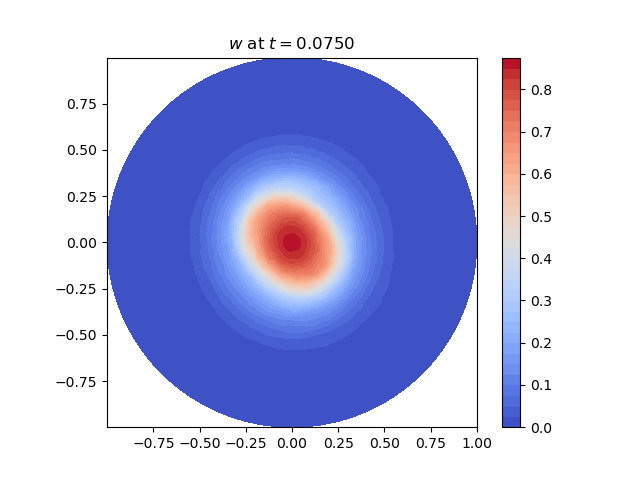}
		\end{minipage}
	\end{tabular}
	\begin{tabular}{l r}
		\rotatebox[origin=c]{90}{t=0.100} &
		\hspace*{-0.5cm}
		\begin{minipage}{0.32\textwidth}
			\centering
			\includegraphics[scale=0.35]{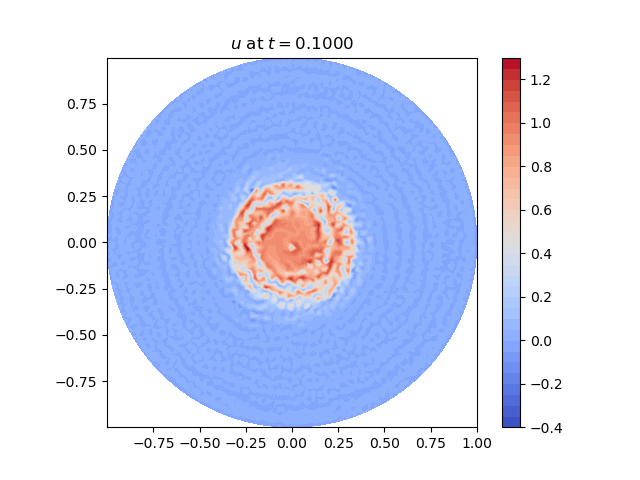}
		\end{minipage}
		\begin{minipage}{0.32\textwidth}
			\centering
			\includegraphics[scale=0.35]{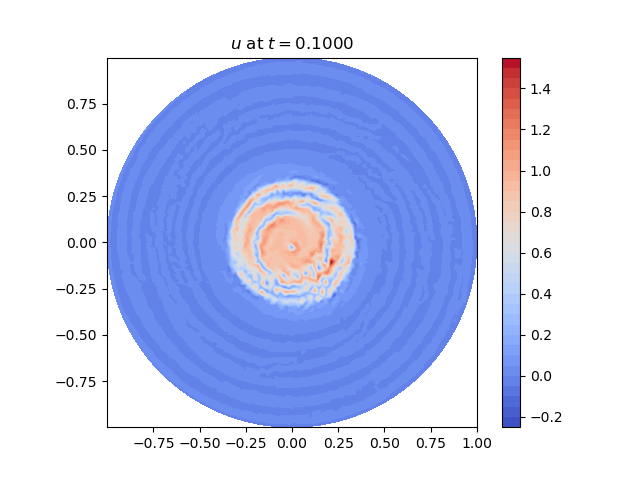}
		\end{minipage}
		\begin{minipage}{0.32\textwidth}
			\centering
			\includegraphics[scale=0.35]{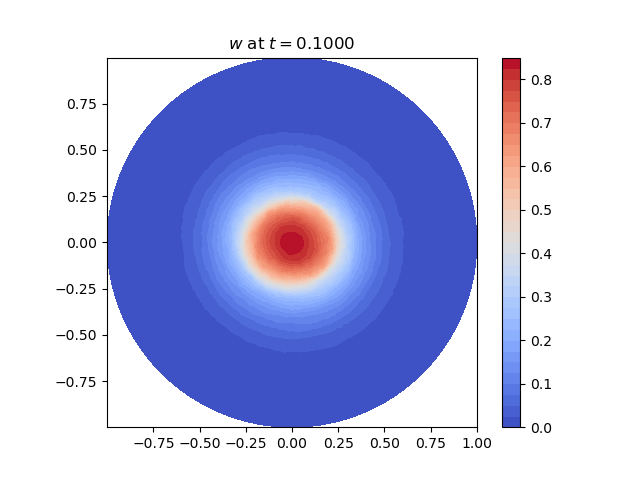}
		\end{minipage}
	\end{tabular}
	\caption{Aggregation of circular regions over time with a strong convection ($\vv=100(y,-x)$).}
	\label{fig:comparacion_esquemas_conveccion}
\end{figure}

\begin{figure}[ht]
	\centering
	\begin{minipage}{0.49\linewidth}
		\centering
		\textbf{Maximum-Minimum}\par\smallskip
		\includegraphics[scale=0.51]{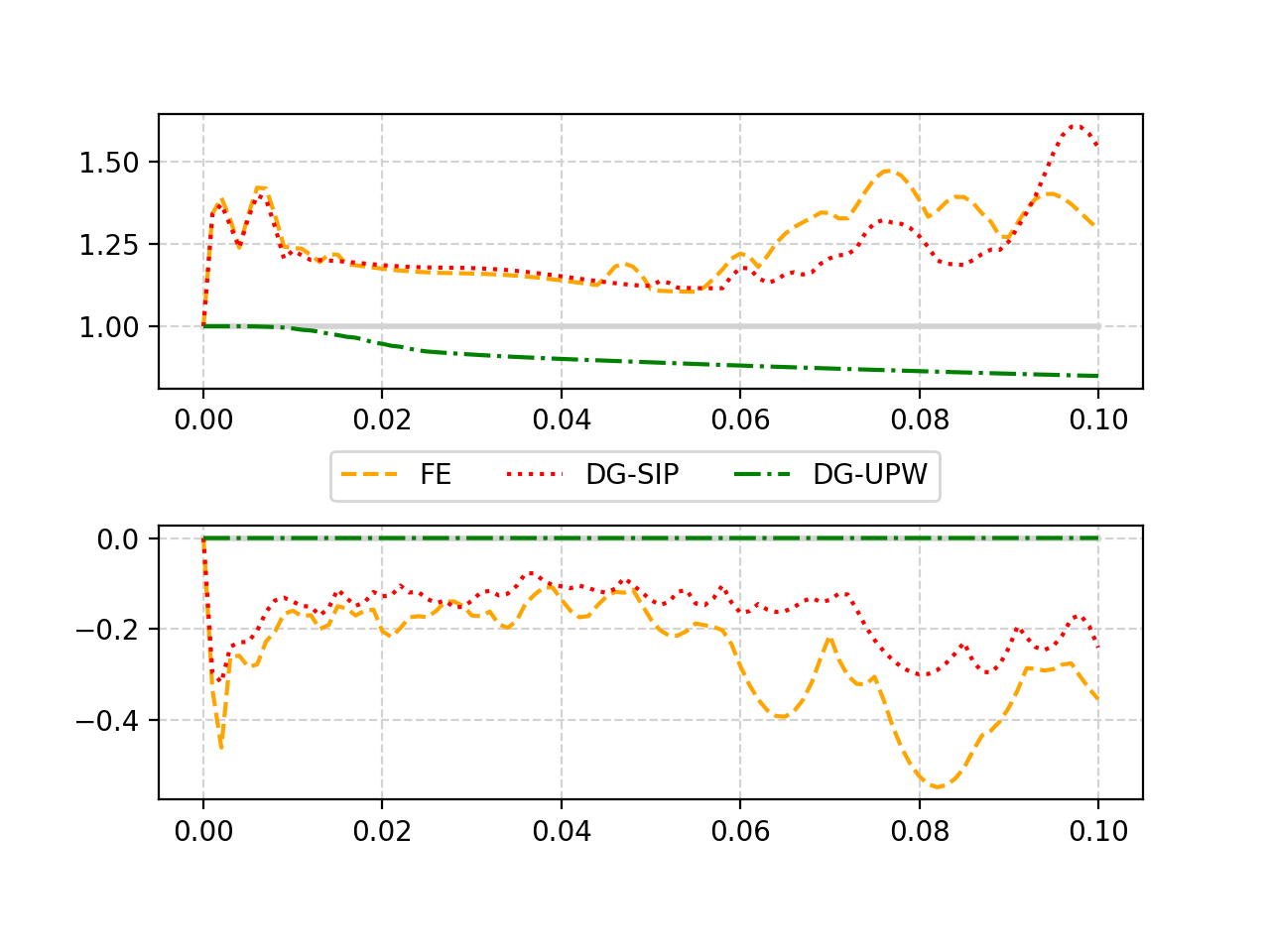}
	\end{minipage}
	\begin{minipage}{0.49\linewidth}
		\centering
		\textbf{Dynamics}
		\includegraphics[scale=0.5]{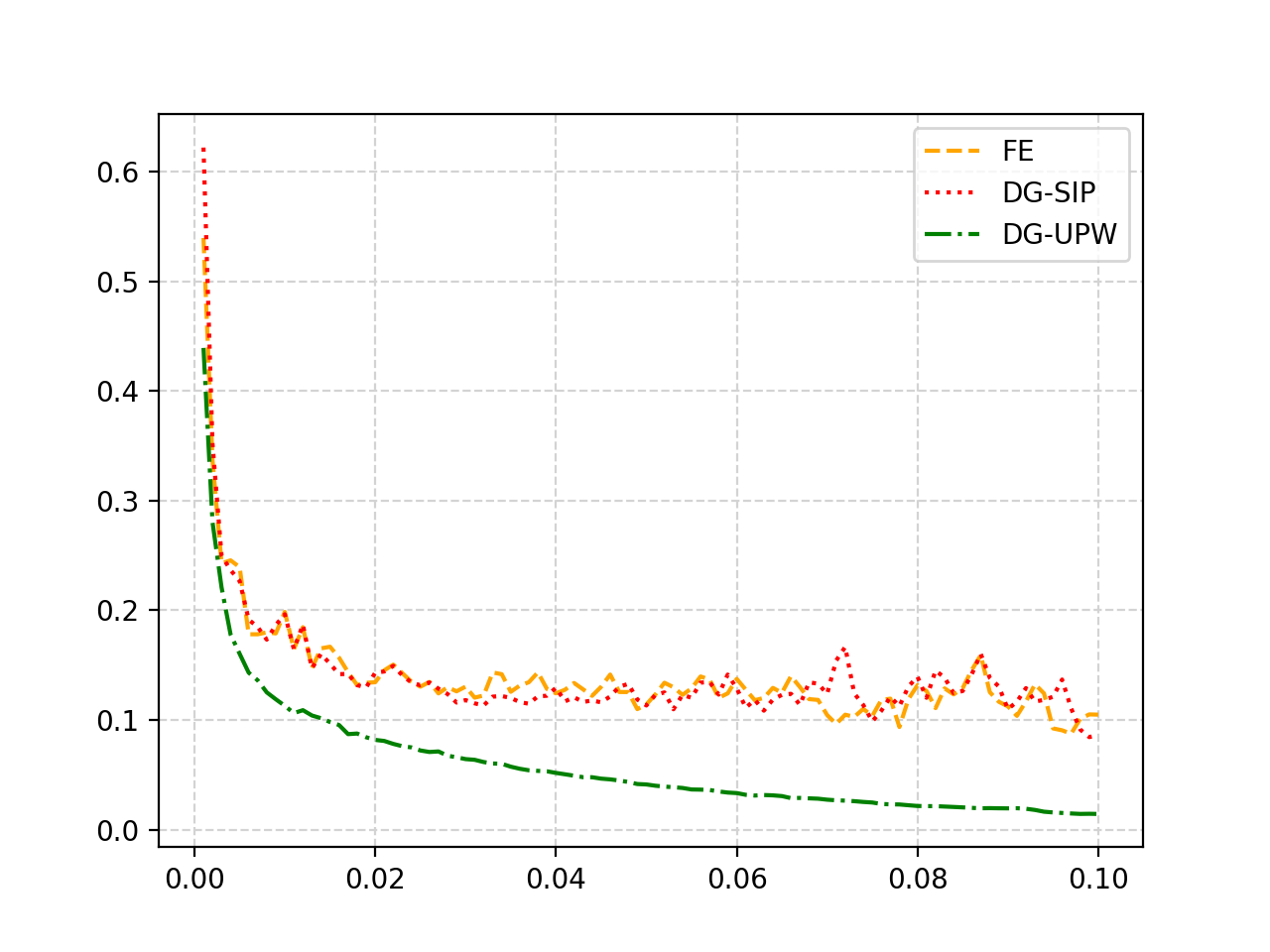}
	\end{minipage}
	\caption{Aggregation of circular phases with strong convection ($\vv=100(y,-x)$). On the left, maximum and minimum of the phase field variable over time. On the right, we plot $\frac{\normaLinf{u^{m+1}-u^m}}{\normaLinf{u^m}}$ to observe the dynamics of the approximations.}
	\label{fig:comparacion_esquemas_conveccion_max}
\end{figure}

\subsubsection{Spinoidal decomposition driven by Stokes cavity flow}

We show the results from a spinoidal decomposition test, in which the initial condition is a small uniformly distributed random perturbation around $0.5$, $u_0(x)\in[0.49,0.51]$ for $x\in\Omega$ as shown in Figure \ref{fig:spinoidal_decomposition_initial_condition}.

\begin{figure}[H]
	\centering
	\textbf{DG-SIP}
	\\
	\begin{minipage}{0.49\linewidth}
		\centering
		\includegraphics[scale=0.51]{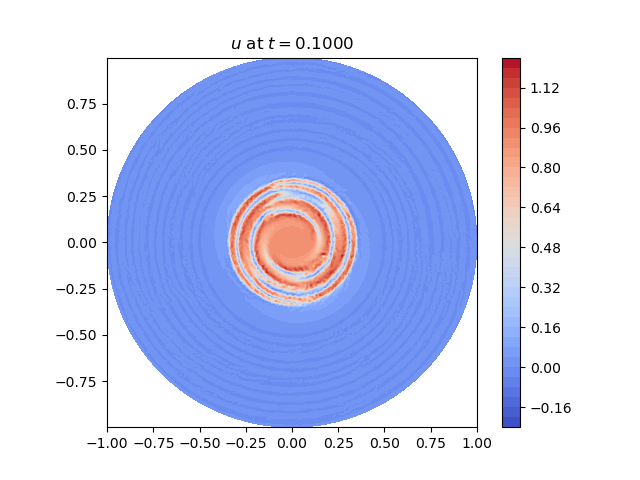}
	\end{minipage}
	\begin{minipage}{0.49\linewidth}
		\centering
		\includegraphics[scale=0.51]{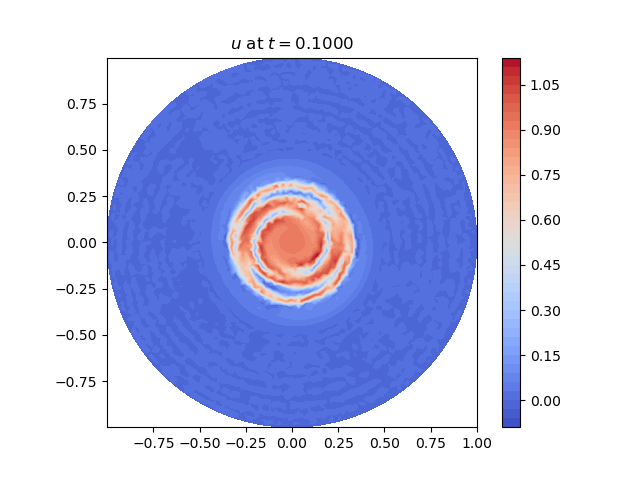}
	\end{minipage}
	\caption{Aggregation of circular phases with strong convection ($\vv=100(y,-x)$) using the DG-SIP scheme. On the left, the result obtained with $h/2\approx2\cdot 10^{-2}$ and $\Pd_1(\T_h)$. On the right, the result obtained with $h\approx4\cdot 10^{-2}$ and $\Pd_2(\T_h)$.}
	\label{fig:comparacion_esquemas_conveccion_refinado}
\end{figure}

\begin{figure}[t]
\centering
\textbf{Random perturbation}\par\smallskip
\includegraphics[scale=0.51]{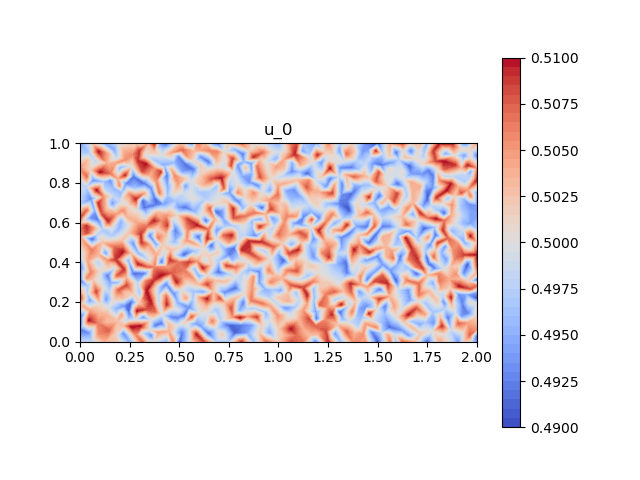}
\caption{Random initial perturbation for the spinoidal decomposition test.}
\label{fig:spinoidal_decomposition_initial_condition}
\end{figure}
As convection vector $\vv$, we take the flow resulting from solving a cavity test for the Stokes equations in the domain $\Omega=[0,2]\times [0,1]$ with Dirichlet boundary conditions given by a parabolic profile
$$\vv(x,y)=(x(2-x),0),\quad\forall (x,y)\in\Gamma_{\text{top}}=\{(x,1)\in\Rset^2\colon x\in[0,2]\}.$$

We used this parabolic profile in order to avoid the discontinuities of
the Stokes velocity for a standard boundary condition $\vv=1$ on $\Gamma_{\text{top}}$,
which produces a non vanishing divergence in the corners where the discontinuities arise.
In fact we have checked that, in this particular case where $\vv=1$ on $\Gamma_{\text{top}}$, the scheme DG-UPW does not preserve the upper bound $u^m\le 1$, although it does preserve the lower bound $0\le u^m$ (recall that, for compressible velocity, positivity is the only property of the solution, see Remark~\ref{rmrk:positividad_modelo_conveccion_lineal}).

We set $\varepsilon=0.005$, $\Delta t = 0.001$, $h\approx 0.07$ and, in this case, we take $\Pe=10$ in order to emphasize the convection effect.
We can observe in the Figures \ref{fig:comparacion_esquemas_spinoidal_decomposition} and \ref{fig:comparacion_esquemas_spinoidal_decomposition_max} how the maximum principle is preserved by our DG-UPW scheme~\eqref{esquema_DG_upw_Eyre_cahn-hilliard+adveccion} while for the other schemes the solution takes values out of the interval $[0,1]$ (to notice that, we must take into account the scale of the values shown on the right-hand side of each picture).

Furthermore, in Figure \ref{fig:comparacion_esquemas_spinoidal_decomposition_max} we can also notice that the approximation obtained using the DG-UPW scheme converges to a stationary state, while it remains in an nonphysical oscillatory state when we use the other schemes.

\begin{figure}[t]
\begin{tabular}{l r}
	\rotatebox[origin=c]{90}{t=10.0} &
	\hspace*{-0.5cm}
	\begin{minipage}{0.32\textwidth}
		\centering
		\textbf{FEM}
		\includegraphics[scale=0.35]{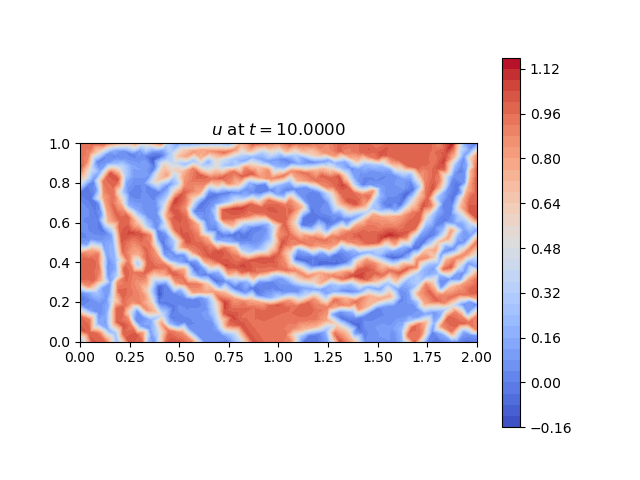}
	\end{minipage}
	\begin{minipage}{0.32\textwidth}
		\centering
		\textbf{DG-SIP}
		\includegraphics[scale=0.35]{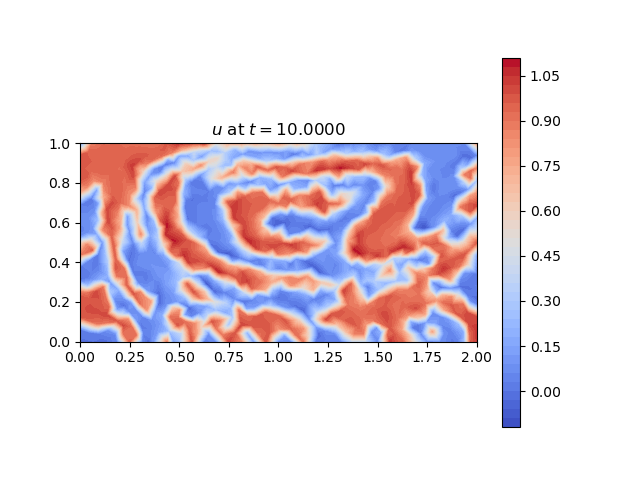}
	\end{minipage}
	\begin{minipage}{0.32\textwidth}
		\centering
		\textbf{DG-UPW}
		\includegraphics[scale=0.35]{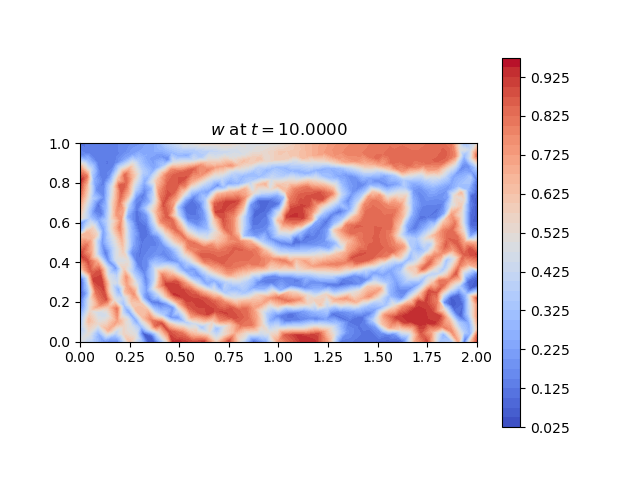}
	\end{minipage}
\end{tabular}
\begin{tabular}{l r}
	\rotatebox[origin=c]{90}{t=20.0} &
	\hspace*{-0.5cm}
	\begin{minipage}{0.32\textwidth}
		\centering
		\includegraphics[scale=0.35]{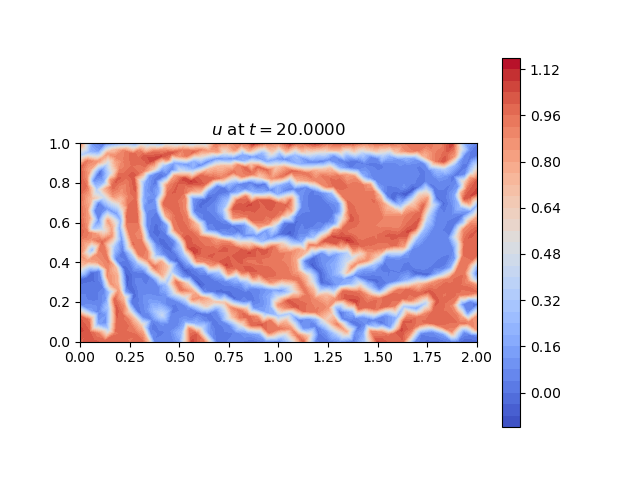}
	\end{minipage}
	\begin{minipage}{0.32\textwidth}
		\centering
		\includegraphics[scale=0.35]{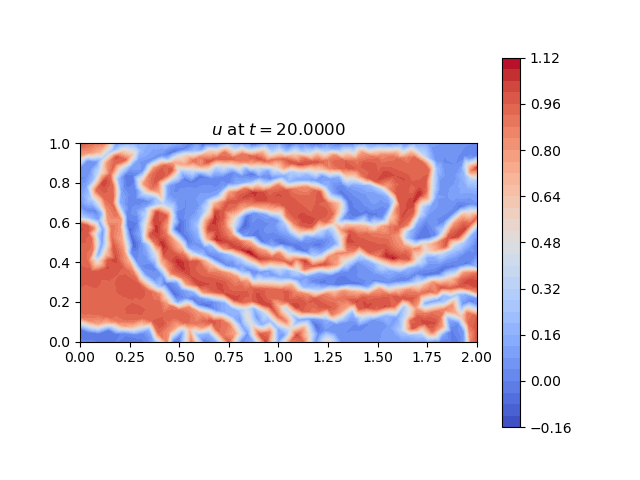}
	\end{minipage}
	\begin{minipage}{0.32\textwidth}
		\centering
		\includegraphics[scale=0.35]{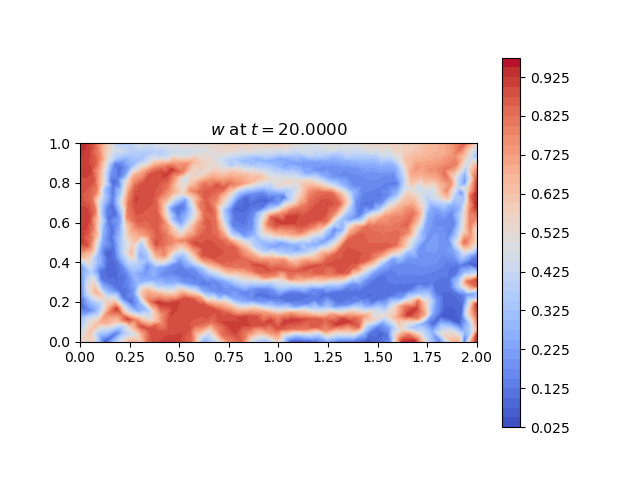}
	\end{minipage}
\end{tabular}
\begin{tabular}{l r}
	\rotatebox[origin=c]{90}{t=50.0} &
	\hspace*{-0.5cm}
	\begin{minipage}{0.32\textwidth}
		\centering
		\includegraphics[scale=0.35]{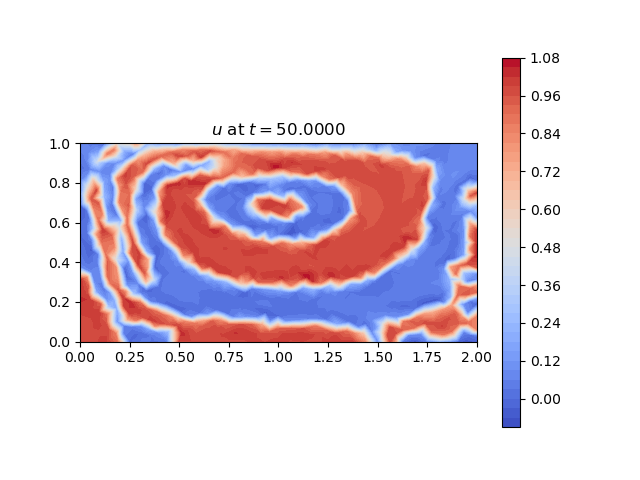}
	\end{minipage}
	\begin{minipage}{0.32\textwidth}
		\centering
		\includegraphics[scale=0.35]{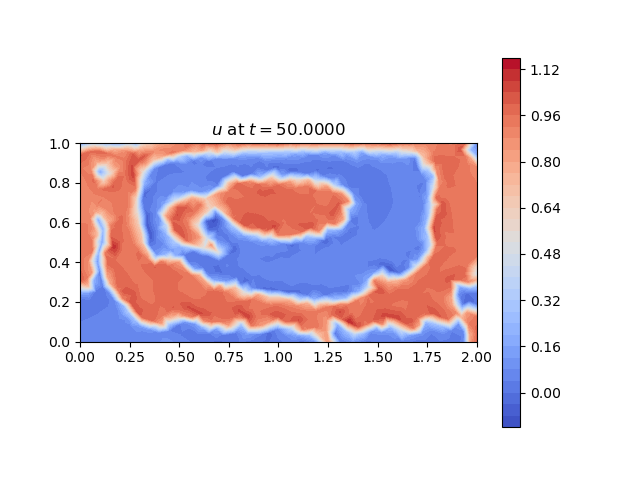}
	\end{minipage}
	\begin{minipage}{0.32\textwidth}
		\centering
		\includegraphics[scale=0.35]{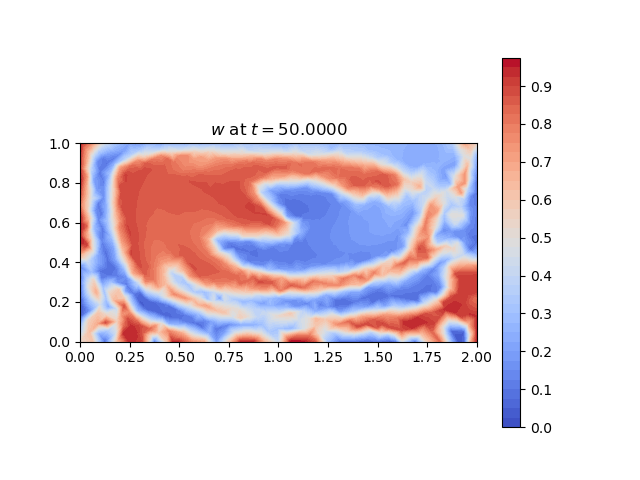}
	\end{minipage}
\end{tabular}
\begin{tabular}{l r}
	\rotatebox[origin=c]{90}{t=100.0} &
	\hspace*{-0.5cm}
	\begin{minipage}{0.32\textwidth}
		\centering
		\includegraphics[scale=0.35]{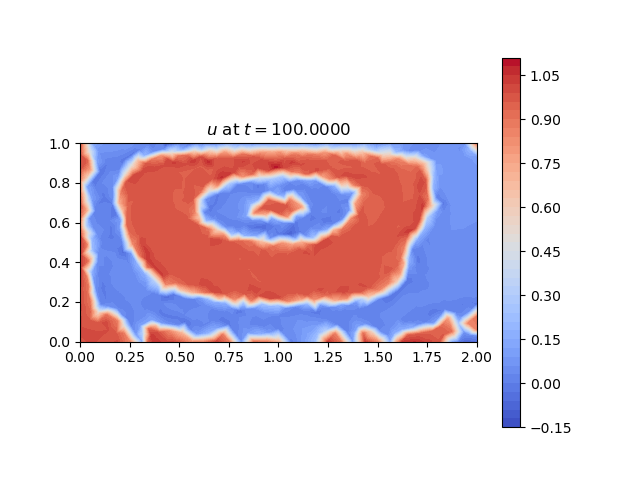}
	\end{minipage}
	\begin{minipage}{0.32\textwidth}
		\centering
		\includegraphics[scale=0.35]{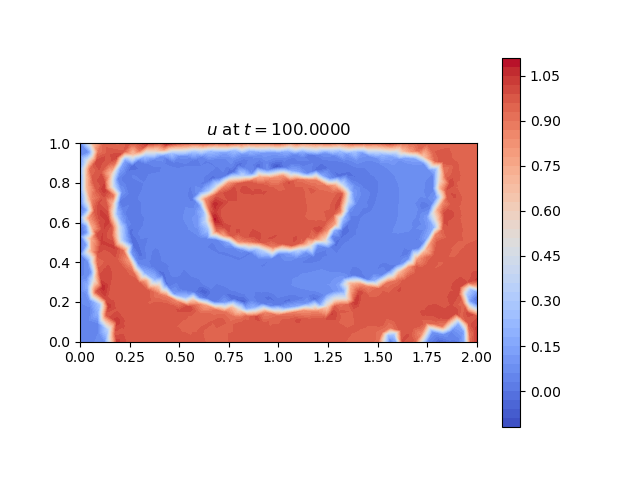}
	\end{minipage}
	\begin{minipage}{0.32\textwidth}
		\centering
		\includegraphics[scale=0.35]{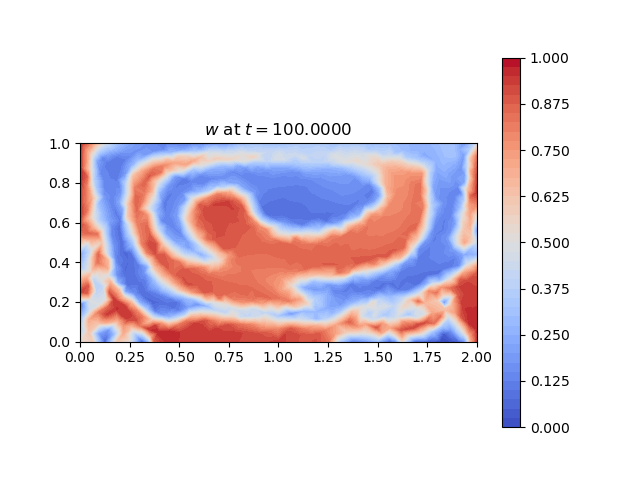}
	\end{minipage}
\end{tabular}
\begin{tabular}{l r}
	\rotatebox[origin=c]{90}{t=400.0} &
	\hspace*{-0.5cm}
	\begin{minipage}{0.32\textwidth}
		\centering
		\includegraphics[scale=0.35]{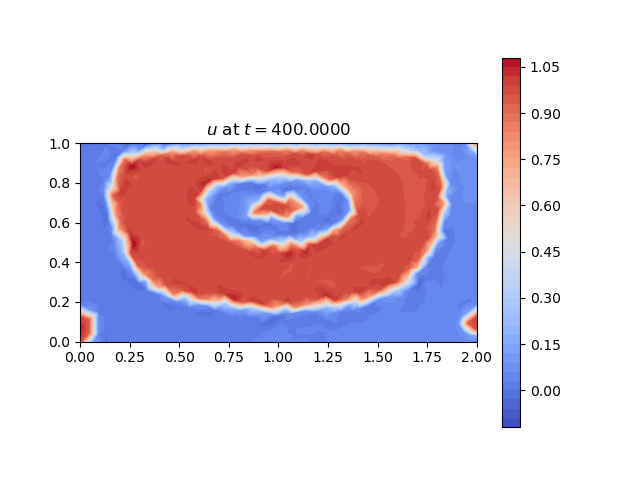}
	\end{minipage}
	\begin{minipage}{0.32\textwidth}
		\centering
		\includegraphics[scale=0.35]{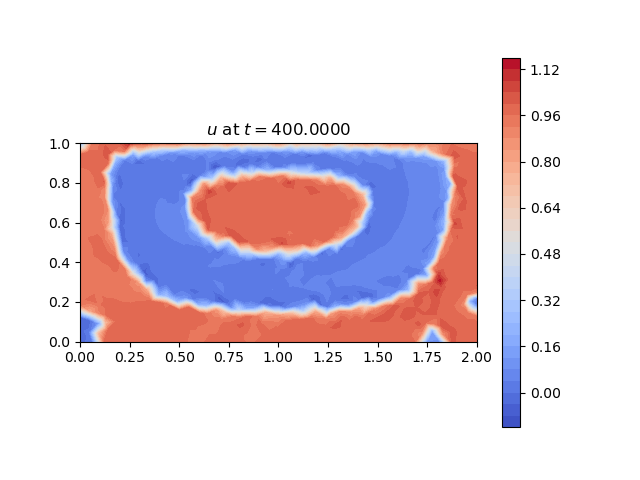}
	\end{minipage}
	\begin{minipage}{0.32\textwidth}
		\centering
		\includegraphics[scale=0.35]{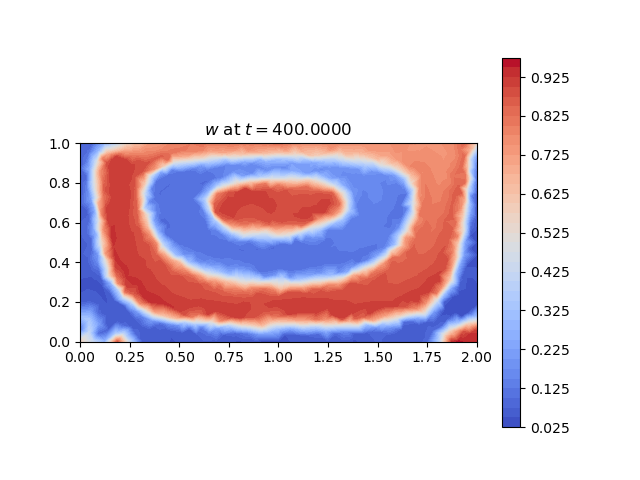}
	\end{minipage}
\end{tabular}
\caption{Spinoidal decomposition over time with convection vector obtained from a cavity test.}
\label{fig:comparacion_esquemas_spinoidal_decomposition}
\end{figure}

\begin{figure}[t]
\centering
\begin{minipage}{0.49\linewidth}
	\centering
	\textbf{Maximum-Minimum}\par\smallskip
	\includegraphics[scale=0.51]{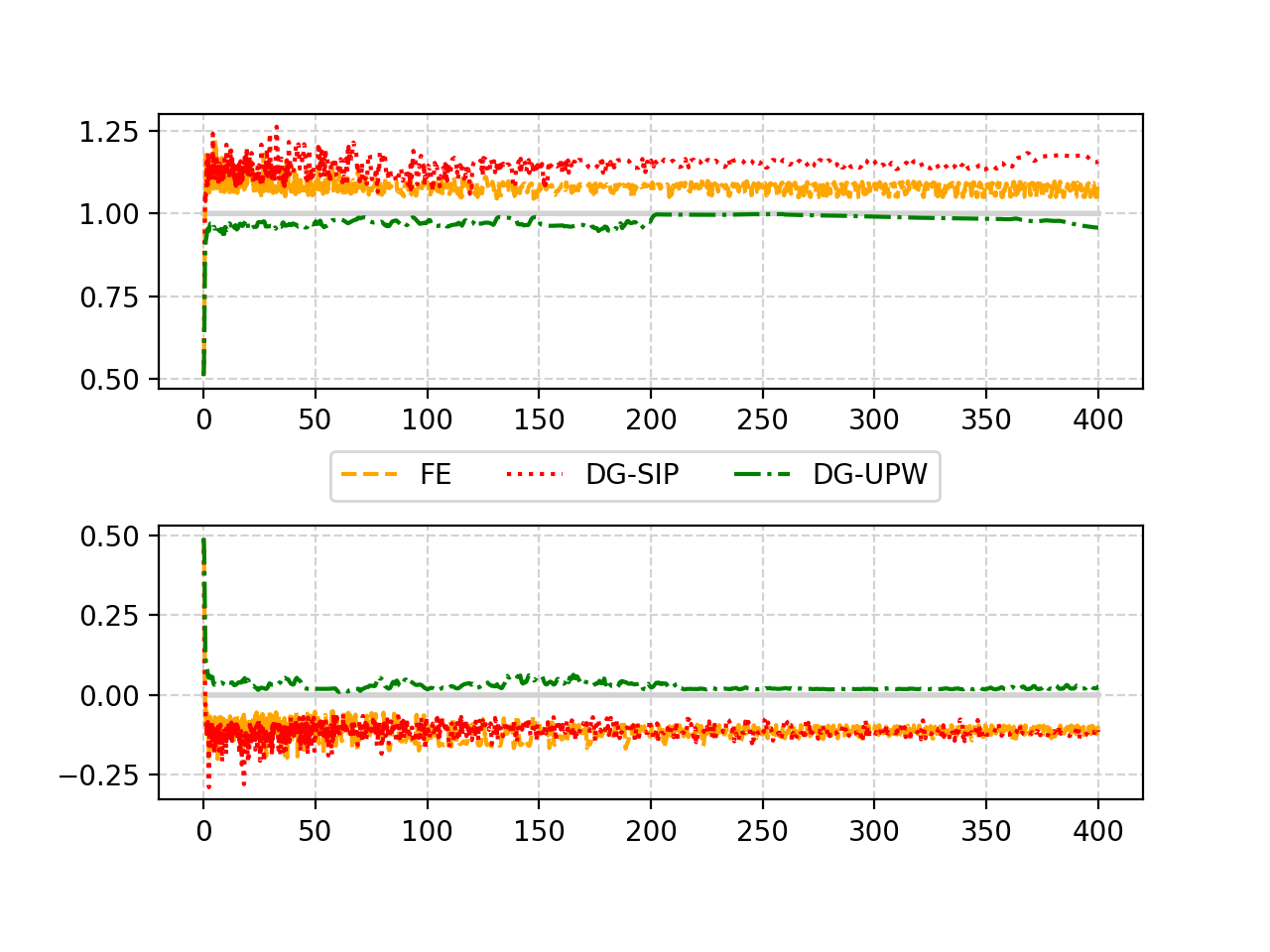}
\end{minipage}
\begin{minipage}{0.49\linewidth}
	\centering
	\textbf{Dynamics}
	\includegraphics[scale=0.5]{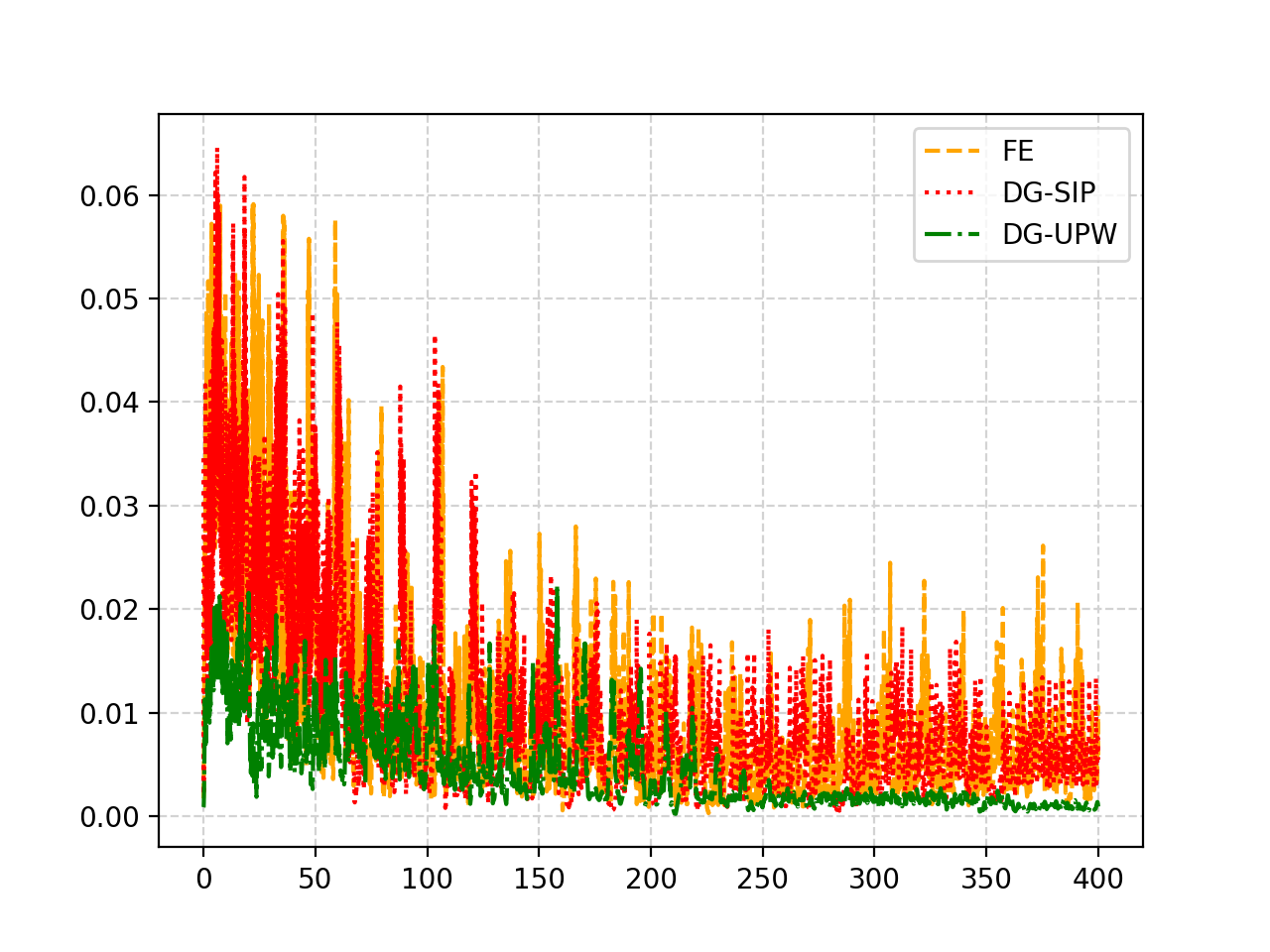}
\end{minipage}
\caption{Spinoidal decomposition with convection vector obtained from a cavity test. On the left, maximum and minimum of the phase field variable over time. On the right, we plot $\frac{\normaLinf{u^{m+1}-u^m}}{\normaLinf{u^m}}$ to observe the dynamics of the approximations.}
\label{fig:comparacion_esquemas_spinoidal_decomposition_max}
\end{figure}

\subsection{Error order test}
\label{sec:error_test}

We now introduce the results of a numerical test in which we
study the convergence order of our numerical scheme DG-UPW, verifying experimentally that the expected order is obtained.
For the {sake} of completeness, we also compare the convergence orders of the two others aforementioned space semidiscretizations: FEM, DG-SIP.
We consider again the same initial conditions than in Subsection \ref{sec:aggregation_circular_regions} with $\varepsilon=0.01$, see Figure~\ref{fig:condicion_inicial_agregacion-analisis_errores}.

First, in the nonconvective case, errors and convergence order
are compared with respect to an approximate solution which is computed using the FEM scheme in a very fine mesh of size
$h=1.414\cdot 10^{-3}$ and a time step $\Delta t=10^{-6}$ (which is taken as the ``exact solution''). In this case we have used conforming structured meshes for the space discretization.
The results for the DG-UPW scheme, which are shown in the first row of Table~\ref{tabla:analisis_errores_no_conveccion}, confirm order 1 (in fact, slightly over 1) in norm $\|\cdot\|_{L^2}$.
These results match our expectations for the $P_0$ approximation of $u^{m+1}$, with the upwind discretization of the nonlinear second-order term.
It is interesting to emphasize that, unexpectedly, the scheme produces kind results in $\|\cdot\|_{H^1}$, reaching order 1.
On the other hand, the FEM and DG-SIP schemes reach order 2 in $L^2$ and order 1 in $H^1$ norms, as expected (see also Table~\ref{tabla:analisis_errores_no_conveccion}).

Next, we focus on the case with convection, where we take $\vv=(y,-x)$ in the unit ball.
The resulting errors and convergence orders computed using the three different schemes over a conforming unstructured mesh are shown in Table \ref{tabla:analisis_errores_conveccion}.
In this case, the errors are computed with respect to the solution obtained for the FEM scheme in a mesh of size $h=4\cdot 10^{-3}$.

In this case, it is interesting to notice that the error order in $\|\cdot\|_{L^2}$ of the DG-UPW scheme is improved and it approaches the order 2 of the other schemes, while order in $\|\cdot\|_{H^1}$ slightly beats the other schemes.

\begin{table}
	\centering
	\footnotesize
	\begin{tabular}{||c|c|c|c|c|c|c|c|c||}
		\hline
		\multirow{2}{*}{Scheme} & \multirow{2}{*}{Norm}& $h\approx 2.8284\cdot 10^{-2}$ & \multicolumn{2}{c|}{$h/2\approx 1.4142\cdot 10^{-2}$} & \multicolumn{2}{c|}{$h/3\approx 9.428\cdot 10^{-3}$} &\multicolumn{2}{c||}{$h/4\approx 7.071\cdot 10^{-3}$} \\
		\cline{3-9}
		&&Error &Error & Order &  Error & Order &  Error & Order \\
		\hline
		\hline
		\multirow{2}{*}{\textbf{DG-UPW}} & $\norma{\cdot}_{L^2}$ & $8.5268\cdot 10^{-3}$ &$3.0933\cdot 10^{-3}$ & $1.46$ &$1.7645\cdot 10^{-3}$  & $1.38$  &$1.2134\cdot 10^{-3}$ & $1.30$   \\
		\cline{2-9}
		& $\norma{\cdot}_{H^1}$ &$8.0000\cdot 10^{-1}$  &$4.0199\cdot 10^{-1}$  & $0.99$ &$2.6081\cdot 10^{-1}$ & $1.07$  &$1.8849\cdot 10^{-1}$ & $1.13$   \\
		\hline
		\hline
                \multirow{2}{*}{FEM}  & $\norma{\cdot}_{L^2}$ &$5.3224\cdot 10^{-3}$  &$1.5679\cdot 10^{-3}$ &$1.76$ &$6.9944\cdot 10^{-4}$  &$1.99$ &$4.0191\cdot 10^{-4}$ & $1.93$ \\
		\cline{2-9}
		& $\norma{\cdot}_{H^1}$ &$8.9963\cdot 10^{-1}$  & $4.1080\cdot 10^{-1}$ & $1.13$  &$2.5252\cdot 10^{-1}$  & $1.2$   &$1.7799\cdot 10^{-1}$  & $1.22$  \\
		\hline
		\multirow{2}{*}{DG-SIP}  & $\norma{\cdot}_{L^2}$ &$4.6466\cdot 10^{-3}$  &$1.3023\cdot 10^{-3}$  & $1.84$  &$5.8945\cdot 10^{-4}$ & $1.96$ &$3.2710\cdot 10^{-4}$ & $2.05$   \\
		\cline{2-9}
		& $\norma{\cdot}_{H^1}$ &$1.1784$  &$5.8331\cdot 10^{-1}$  & $1.01$  &$3.6254\cdot 10^{-1}$ & $1.17$  &$2.6024\cdot 10^{-1}$ &$1.15$ \\
		\hline
	\end{tabular}
	\captionof{table}{Errors and convergence orders in $T=0.001$ without convection ($\vv=0$).}
	\label{tabla:analisis_errores_no_conveccion}
\end{table}

\begin{table}
	\centering
	\footnotesize
	\begin{tabular}{||c|c|c|c|c|c|c|c|c||}
		\hline
		\multirow{2}{*}{Scheme} & \multirow{2}{*}{Norm}& $h\approx 4\cdot 10^{-2}$ & \multicolumn{2}{c|}{$h/2\approx 2\cdot 10^{-2} $} & \multicolumn{2}{c|}{$h/3\approx 1.3333\cdot 10^{-2}$} &\multicolumn{2}{c||}{$h/4\approx 1\cdot 10^{-2}$} \\
		\cline{3-9}
		&&Error &Error & Order &  Error & Order &  Error & Order \\
		\hline
		\hline
		\multirow{2}{*}{\textbf{DG-UPW}} & $\norma{\cdot}_{L^2}$ &$1.7288\cdot 10^{-2}$  &$6.9446\cdot 10^{-3}$ & $1.32$  &$3.3102\cdot 10^{-3}$ & $1.83$  &$2.0578\cdot 10^{-3}$ & $1.65$   \\
		\cline{2-9}
		& $\norma{\cdot}_{H^1}$ &$1.4549$ &$6.0305\cdot 10^{-1}$ & $1.27$ &$3.0204\cdot 10^{-1}$ & $1.71$ &$2.0315\cdot 10^{-1}$ & $1.38$  \\
		\hline
		\hline
		\multirow{2}{*}{FEM}  & $\norma{\cdot}_{L^2}$ &$6.8347\cdot 10^{-3}$  &$2.1213\cdot 10^{-3}$  & $1.69$  &$9.7749\cdot 10^{-4}$ & $1.91$  &$5.3883\cdot 10^{-4}$ & $2.07$ \\
		\cline{2-9}
		& $\norma{\cdot}_{H^1}$ &$8.3104\cdot 10^{-1}$  &$3.8060\cdot 10^{-1}$  & $1.13$  &$2.1887\cdot 10^{-1}$ & $1.36$ &$1.4991\cdot 10^{-1}$ & $1.32$  \\
		\hline
		\multirow{2}{*}{DG-SIP}  & $\norma{\cdot}_{L^2}$ &$6.5242\cdot 10^{-3}$  & $1.9557\cdot 10^{-3}$  & $1.74$ &$8.9471\cdot 10^{-4}$  & $1.93$ &$5.0257\cdot 10^{-4}$ & $2.00$ \\
		\cline{2-9}
		& $\norma{\cdot}_{H^1}$ &$1.1980$ &$6.1624\cdot 10^{-1}$  & $0.96$ &$3.8451\cdot 10^{-1}$  & $1.16$ &$2.7439\cdot 10^{-1}$ &$1.17$ \\
		\hline
	\end{tabular}
	\captionof{table}{Errors and convergence orders in $T=0.001$ with convection ($\vv=(y,-x)$).}
	\label{tabla:analisis_errores_conveccion}
\end{table}

As a technical comment, notice that, in order to compute the errors, we projected on a $\Pc_1$ space both the exact and the DG solution obtained when using the DG-SIP. In the case of the DG-UPW scheme we have taken $w$ as the continuous solution.

\section*{Acknowledgments}
The first author has been supported by \textit{UCA FPU contract UCA/REC14VPCT/2020 funded by Universidad de Cádiz} and by a \textit{Graduate Scholarship funded by the University of Tennessee at Chattanooga}. The second and third authors have been supported by Grant PGC2018-098308-B-I00 by MCI N/AEI/ 10.13039/501100011033 and by ERDF a way of making Europe.

\bibliography{biblio_database}

\end{document}